\documentclass[english,11pt,reqno]{amsart}
\usepackage{geometry}
\usepackage[T1]{fontenc}
\usepackage[utf8]{inputenc}
\usepackage{newtxtext,microtype}
\usepackage{amsmath,amssymb,amsthm,mathtools,mathrsfs,amscd}
\usepackage{newtxmath}
\usepackage{bm,enumitem,booktabs,graphicx,tikz,array,xcolor}
\usepackage{aliascnt}
\usepackage{hyperref}
\usepackage[hyperpageref]{backref}

\usepackage[nameinlink,noabbrev]{cleveref}
\numberwithin{equation}{section}
\allowdisplaybreaks[4]
\setlist[itemize]{leftmargin=2.2em,itemsep=0.25em,topsep=0.35em}
\setlist[enumerate]{leftmargin=2.4em,itemsep=0.25em,topsep=0.35em}
\crefname{section}{Section}{Sections}
\Crefname{section}{Section}{Sections}
\crefname{subsection}{Subsection}{Subsections}
\Crefname{subsection}{Subsection}{Subsections}

\newcommand{\C}{\mathbf C}
\newcommand{\R}{\mathbf R}

\newcommand{\Z}{\mathbf Z}
\newcommand{\Pone}{\mathbf P^1}

\newcommand{\PP}{\mathbf P}
\newcommand{\cA}{\mathcal A}
\newcommand{\cO}{\mathcal O}
\newcommand{\OO}{\mathcal O}

\newcommand{\GL}{\mathrm{GL}}
\newcommand{\Hom}{\mathrm{Hom}}

\newcommand{\Spec}{\mathrm{Spec}}

\newcommand{\red}{\mathrm{red}}

\newcommand{\tor}{\mathrm{tor}}

\newcommand{\e}{\mathrm e}
\newcommand{\ii}{\mathrm i}

\newcommand{\pr}{\operatorname{pr}}

\DeclareMathOperator{\Bl}{Bl}

\hypersetup{colorlinks=true,linkcolor=blue!55!black,citecolor=blue!62!black,
urlcolor=blue!80!black,
pdftitle={Oka Geometry of Complex Six-Spheres},
pdfauthor={}}
\newtheorem{theorem}{Theorem}[section]
\crefname{theorem}{Theorem}{Theorems}
\Crefname{theorem}{Theorem}{Theorems}
\theoremstyle{plain}
\newaliascnt{proposition}{theorem}
\newtheorem{proposition}[proposition]{Proposition}
\aliascntresetthe{proposition}
\crefname{proposition}{Proposition}{Propositions}
\Crefname{proposition}{Proposition}{Propositions}
\newaliascnt{lemma}{theorem}
\newtheorem{lemma}[lemma]{Lemma}
\aliascntresetthe{lemma}
\crefname{lemma}{Lemma}{Lemmas}
\Crefname{lemma}{Lemma}{Lemmas}
\newaliascnt{corollary}{theorem}
\newtheorem{corollary}[corollary]{Corollary}
\aliascntresetthe{corollary}
\crefname{corollary}{Corollary}{Corollaries}
\Crefname{corollary}{Corollary}{Corollaries}
\theoremstyle{definition}
\newaliascnt{definition}{theorem}
\newtheorem{definition}[definition]{Definition}
\aliascntresetthe{definition}
\crefname{definition}{Definition}{Definitions}
\Crefname{definition}{Definition}{Definitions}
\theoremstyle{remark}
\newaliascnt{remark}{theorem}

\aliascntresetthe{remark}
\crefname{remark}{Remark}{Remarks}
\Crefname{remark}{Remark}{Remarks}
\crefname{equation}{Equation}{Equations}
\Crefname{equation}{Equation}{Equations}
\crefname{figure}{Figure}{Figures}
\Crefname{figure}{Figure}{Figures}
\crefname{table}{Table}{Tables}
\Crefname{table}{Table}{Tables}
\title{Oka Geometry of Complex Six-Spheres}
\author{Yun-Heng Du}
\address[Yun-Heng Du]{Academy of Mathematics and Systems Science,
	Chinese Academy of Sciences, Beijing 100190, China}
\email{duyunheng@amss.ac.cn}
\author{Bin Guo}
\address[Bin Guo]{Academy of Mathematics and Systems Science,
	Chinese Academy of Sciences, Beijing 100190, China}
\email{guobin@amss.ac.cn}
\author{Yuta Kusakabe}
\address[Yuta Kusakabe]{Faculty of Mathematics, Kyushu University,
	744 Motooka, Nishi-ku, Fukuoka 819-0395, Japan}
\email{kusakabe@math.kyushu-u.ac.jp}
\author{Song-Yan Xie}
\address[Song-Yan Xie]{State Key Laboratory of Mathematical Sciences,
	Academy of Mathematics and Systems Science, Chinese Academy of Sciences,
	Beijing 100190, China; School of Mathematical Sciences, University of
	Chinese Academy of Sciences, Beijing 100049, China}
\email{xiesongyan@amss.ac.cn}
\date{}
\begin{document}

\begin{abstract}
Alp\"oge constructed a complex structure on $S^6$ through a torus
fibration with exactly three special fibres. We prove that every member
of the associated admissible period family is Oka and that the family
projection over each relatively compact parameter disc is an Oka map.
The Oka property survives the deletion or simultaneous blow-up of any
finite set of distinct points in any member. Blowing up a single point
therefore produces an Oka complex manifold diffeomorphic, but not
biholomorphic, to $\PP^3$. The proofs combine explicit cyclic coverings
with analytic extension and the invariance of the Oka property under
holomorphic coverings.
\end{abstract}


\maketitle

\setcounter{tocdepth}{2}
\begingroup
\hypersetup{linkcolor=black}
\tableofcontents
\endgroup

\section{Introduction}\label{nm:sec:introduction}

The Oka principle begins with a flexibility question: when can a
continuous map from a Stein space be deformed to a holomorphic map?
Informally, a complex manifold is \emph{Oka} if these deformations can be
made with approximation, interpolation, and parameters
\cite{ForstnericBook,ForstnericSurvey}. Gromov proved, using holomorphic
sprays, that elliptic manifolds are Oka \cite{GromovOka}. For a
holomorphic submersion, an \emph{Oka map} is the relative version:
liftings over a fixed map to the base enjoy the same flexibility, and
the submersion is a Serre fibration \cite{KusakabeMaps}.

Kusakabe's analytic-Zariski localization theorem
\cite{KusakabeLocalization} assembles the Oka property from an
analytic-Zariski open cover. His Oka-map criterion extends this
strategy to holomorphic submersions \cite{KusakabeMaps}. These tools
are well suited to a fibration whose general fibre is homogeneous
while its special fibres are singular or multiple.

Alp\"oge constructed a torus-fibred complex threefold diffeomorphic
to $S^6$ \cite{AlpogeS6}. Its general fibre is a complex two-torus,
and the fibration has exactly three special values on $\Pone$: two multiple
fibres with bielliptic reductions and one irreducible nonnormal cusp
fibre, whose normalization is a degree-six del Pezzo surface. Any
ordered triple of distinct values on $\Pone$ can be moved to
$0,1,\infty$ by an automorphism; these coordinates are used
throughout the paper. The orders $3$ and $4$ and the nature of the cusp remain
essential to the construction.
We use Alp\"oge's construction and smooth-type theorem, including
its complex structure on $S^6$; Appendix~\ref{app:source-six-sphere} records
the particular source ingredients used here and their locations.

The smooth fibres and $\Pone$ are Oka, so it is natural to ask whether
the total space is Oka. The base of the smooth fibration,
$\Pone\setminus\{0,1,\infty\}$, is hyperbolic; recovering flexibility
therefore depends on the three special fillings. The proof passes to cyclic
covers containing explicit two-row conic bands, obtained by retaining
two adjacent toric rows in the cusp model. Their algebraic charts
provide Oka open subsets. The analytic extension theorem
\cite[Theorem~1.9]{AnDuGuoWangXieLAP} extends the Oka property from
such a nonempty analytic-Zariski open subset to its ambient manifold.

An additive period direction determines the admissible half-plane
$\cA\subset\C$ in \cref{prop:admissible-half-plane}. Over each relatively
compact parameter disc $\Delta\Subset\cA$ the fixed
filling data give a holomorphic family
$\pi_\Delta:\mathfrak X_\Delta\to\Delta$ of compact threefolds
$X_a$. Here $U$ is the uniformizing upper half-plane of the base and
$\tau,\mu,\beta_0$ are the holomorphic period functions on $U$ constructed in
\cref{sec:data}. The marked matrix is
\[
 \Omega_a(z)=\begin{pmatrix}
 6\mu(z)&\tau(z)&1&0\\
 \beta_0(z)+a&\mu(z)&0&1
 \end{pmatrix}.
\]
The construction fixes the monodromy, finite logarithmic twists and
untranslated cusp gluing as $a$ varies. Its proper submersion and
smooth-type comparison are proved in
\cref{sec:relative-compactification}. Write
$f_a:X_a\to\Pone$ for the torus fibration and
$S_{1,a},S_{2,a}$ for the two finite reductions. The special fibres are
\begin{equation}\label{nm:eq:multiple-fibres}
 f_a^*(0)=3S_{1,a},\qquad f_a^*(1)=4S_{2,a},\qquad
 f_a^{-1}(\infty)=W_a.
\end{equation}

\begin{theorem}\label{nm:thm:oka-local}
For every disc $\Delta\Subset\cA$, the projection
$\pi_\Delta:\mathfrak X_\Delta\to\Delta$ is an Oka map. For every
$a\in\cA$, the manifold $X_a$ and both
$X_a\setminus S_{1,a}$ and $X_a\setminus S_{2,a}$ are Oka.
\end{theorem}

Deleting either finite multiple fibre leaves one multiple fibre, which
is removed by a finite root base change. A cyclic period-lattice cover
then contains a two-row conic band over an elliptic surface. Removing
a section from that surface gives an explicit Danielewski model.
Over a two-dimensional algebraic torus in this model, the principal
conic open is a holomorphically trivial $\C^*$-bundle, hence isomorphic
to $(\C^*)^3$.
Analytic extension gives the Oka property of the band and the cyclic
cover. Classical invariance under holomorphic coverings
\cite[Proposition~2.8]{ForstnericSurvey} gives the finite-fibre
complement, and a final analytic extension gives $X_a$.

The family assertion requires more than fibrewise Oka. The chosen
period-lattice kernel has periods independent of $a$, so its covering
family over a disc is a product with one fixed Oka manifold. The
parametric Oka property of the product projection passes through this
covering by Kusakabe's invariance theorem
\cite[Corollary~5.6]{KusakabeMaps}. Localization then gives the Oka property
of $\pi_\Delta$. The cyclic-cover construction and relative argument appear
in \cref{nm:sec:relative-oka}.

Chen~\cite{ChenConditionalOka} independently established the Oka
property for the admissible members of Alp\"oge's period family.

Analytic extension also applies to the torus-fibred complex
six-spheres of Liu--Rollenske \cite{LiuRollenskeS6} and
Viaclovsky \cite{ViaclovskyTwoParameterS6}. In each construction,
deleting the unique multiple fibre leaves a proper fibration over
$\C$ with a holomorphic section and connected complex two-tori as its
smooth general fibres. The torus-fibration theorem
\cite[Theorem~7.1]{AnDuGuoWangXieUbiquitous} makes this complement
Oka. Analytic extension \cite[Theorem~1.9]{AnDuGuoWangXieLAP}
then gives the Oka property of the compact total space.

Removing or blowing up points tests the flexibility of this geometry.
For a finite centre, its complete lift to a cyclic cover meets each two-row
band in two parts: finitely many full deck orbits in the common core and
finitely many toroidal points.

\begin{theorem}\label{thm:intro-point-modifications}
For every $a\in\cA$ and every finite set $P\subset X_a$ of distinct points,
both $X_a\setminus P$ and the simultaneous blow-up
$\operatorname{Bl}_P X_a$ are Oka.
\end{theorem}

On one principal band chart, a proper monomial character makes the
logarithmic lift of this complete centre tame. Its complement is Oka.
Analytic extension fills the rest of the band and the cyclic cover.
Invariance under holomorphic coverings, followed by analytic extension,
gives $X_a\setminus P$. The exceptional-divisor
complement in $\operatorname{Bl}_P X_a$ is precisely this punctured
manifold, so one further analytic extension proves the blow-up result.
The same principal-open argument yields the band curve complements in
\cref{sec:point-complements}.

For a single point $p$, the Oka manifold $\Bl_pX_a$ is
diffeomorphic to $\PP^3$. Its complex structure differs from the
standard projective one: the first Chern class of its tangent bundle
has divisibility $2$, whereas that of the standard $\PP^3$ has
divisibility $4$. The topology and Chern-class calculation are
spelled out in \cref{cor:exotic-projective-threefold}.

The band geometry also supplies curve complements. If any chosen
subcollection of the canonical exceptional curves is deleted from a
two-row band, the result remains Oka, even together with the mixed
finite/deck point deletions used above. The complement of the complete
residual-axis curve configuration is Oka as well
(\cref{thm:partial-exceptional,thm:axis-deletion}). These are band
statements.

The Oka--$1$ property is the corresponding Runge approximation
and Weierstrass jet-interpolation property for maps from open Riemann
surfaces \cite[Definition~1.1]{ForstnericLarussonOka1}. For an arbitrary
smooth compact curve $C\subset X_a$, the blow-up theorem of
Forstneri\v c--L\'arusson gives $\operatorname{Bl}_C X_a$ this property
(\cref{cor:smooth-curve-oka1}); the local $(-1,-1)$ conifold model
has an algebraically elliptic blow-up
(\cref{prop:local-conifold-curve}).

The paper is organized as follows. \Cref{sec:relative-compactification}
constructs the period family and its fillings.
\Cref{nm:sec:relative-oka} develops the cyclic bands, identifies their
algebraic charts, and proves the absolute and relative Oka statements.
\Cref{sec:point-complements} uses these charts for finite-point
deletions, simultaneous point blow-ups, and band curve complements.
Appendix A contains the Hausdorff gluing and
choice-independence verifications, while Appendix B records the ingredients
of Alp\"oge's complex-six-sphere construction that enter this paper.

\section{Period data and relative compactification}
\label{sec:relative-compactification}
\subsection{Period data and the smooth family}
\label{sec:data}

Throughout, all complex spaces are reduced unless otherwise stated.

A discrete action on a Hausdorff space is called properly discontinuous if,
for every pair of compact subsets, only finitely many group elements carry
the first to meet the second.

Let
\[
 B=\Pone,\qquad
 p_1=0,\quad p_2=1,\quad p_0=\infty,\qquad
 B^\circ=B\setminus\{p_0,p_1,p_2\}.
\]
We use Alp\"oge's oriented $(3,4,\infty)$ uniformization
\cite[Proposition~2.11(i)--(iii)]{AlpogeS6}: a copy $U$ of the upper
half-plane carries a properly discontinuous Fuchsian action of
\[
 \Gamma=\langle g_1,g_2\mid g_1^3=g_2^4=1\rangle,
\]
with fixed points $z_1,z_2\in U$ of $g_1,g_2$, respectively, such that
$U/\Gamma\simeq B\setminus\{p_0\}$, with finite orbifold orders $3$
and $4$ at $p_1,p_2$.  The element $g_0=(g_1g_2)^{-1}$ is a primitive
parabolic generator of the cusp stabilizer.  On
\[
 U^\circ=U\setminus(\Gamma z_1\cup\Gamma z_2)
\]
the action is free, and the orbit map
\begin{equation}
\label{eq:core-cover}
 \varpi:U^\circ\longrightarrow B^\circ
\end{equation}
is a holomorphic covering with deck group $\Gamma$.

Let
\[
 \Lambda=\Z\langle \widehat\gamma,\widehat u,
                      \widehat w,\widehat\delta\rangle.
\]
Write $\Lambda_\R=\Lambda\otimes_\Z\R$.
The homological monodromy matrices in this ordered basis are
\begin{equation}
\label{eq:A12}
 A_1=
 \begin{pmatrix}
 1&0&0&0\\
 6&0&1&0\\
 -6&-1&-1&0\\
 -2&1&0&1
 \end{pmatrix},
 \qquad
 A_2=
 \begin{pmatrix}
 1&0&0&0\\
 0&0&-1&0\\
 -6&1&0&0\\
 3&0&1&1
 \end{pmatrix}.
\end{equation}
Direct multiplication gives $A_1^3=A_2^4=I$.  The displayed presentation of
$\Gamma$ therefore yields a representation
$A:\Gamma\to\GL(\Lambda)$ by $A(g_j)=A_j$; its representation law is
$A(gh)=A(g)A(h)$.
The first row of each matrix in \eqref{eq:A12} is $(1,0,0,0)$.
Write $\gamma:\Lambda\to\Z$ for this invariant first-coordinate
functional.

\begin{proposition}
\label{prop:period-data-existence}
There are holomorphic functions
\[
 \tau,\mu,\beta_0:U\longrightarrow\C,
 \qquad \Im\tau>0,
\]
and holomorphic matrices $R_g:U\to\GL(2,\C)$ satisfying
\begin{align}
 R_{gh}(z)&=R_g(hz)R_h(z),
 \label{eq:R-cocycle}\\
 \Omega_a(gz)&=R_g(z)\Omega_a(z)A(g)^{-1}
 \label{eq:period-equivariance}
\end{align}
for every $g,h\in\Gamma$ and every $a\in\C$, where
\begin{equation}
\label{eq:period-matrix}
 \Omega_a(z)=
 \begin{pmatrix}
 6\mu(z)&\tau(z)&1&0\\
 \beta_0(z)+a&\mu(z)&0&1
 \end{pmatrix}.
\end{equation}
Let $t_c$ be the pullback to a sufficiently small lifted cusp neighbourhood
of Alp\"oge's local coordinate at $p_0$. There is a
holomorphic unit $u_1$ at $t_c=0$ such that
\begin{equation}
 \e^{2\pi\ii\tau}=t_cu_1(t_c).
 \label{eq:cusp-period-unit}
\end{equation}
After choosing a holomorphic logarithm of $u_1$, set
\begin{equation}
 h(t_c)=\frac{1}{2\pi\ii}\log u_1(t_c),
 \qquad s=\tau-h(t_c).
 \label{eq:cusp-corrected-period}
\end{equation}
Then $\e^{2\pi\ii s}=t_c$, and both $\mu$ and $\beta_0+\tau$ extend
holomorphically to $t_c=0$.

Set
\begin{equation}
\label{eq:D-a}
 \mathscr D_a(z)=\Im(\beta_0(z)+a)
        -\frac{6(\Im\mu(z))^2}{\Im\tau(z)}.
\end{equation}
The function $\mathscr D_0$ is
$\Gamma$-invariant, descends to $B\setminus\{p_0\}$, and tends to $-\infty$
as $p_0$ is approached.  Its descended function is continuous at the two
finite orbifold points; together with the cusp limit, this makes
$\mathscr D_0$ bounded above.
\end{proposition}

\begin{proof}
In the ordered basis
$(\widehat\gamma,\widehat u,\widehat w,\widehat\delta)$ of
\eqref{eq:A12}, Alp\"oge~\cite[Theorem~3.4(i)--(v)]{AlpogeS6}
constructs $\tau,\mu$ and a solution $\beta_0$, unique up to an additive
constant, with the equivariance for $a=0$. Parts~(i)--(iii) give
\eqref{eq:cusp-period-unit}--\eqref{eq:cusp-corrected-period} and the stated
extensions of $\mu$ and $\beta_0+\tau$. In the present generator convention
the transformation laws are
\begin{equation}\label{eq:period-generator-laws}
 \tau(g_1z)=\frac{\tau(z)-1}{\tau(z)},\qquad
 \tau(g_2z)=-\frac1{\tau(z)},\qquad
 \tau(g_0z)=\tau(z)-1,
\end{equation}
\begin{equation}\label{eq:mu-generator-laws}
 \mu(g_1z)=\frac{1-\mu(z)}{\tau(z)},\qquad
 \mu(g_2z)=1+\frac{\mu(z)}{\tau(z)}.
\end{equation}
With $\varphi_1=2-6(1-\mu)^2/\tau$ and
$\varphi_2=-3-6\mu^2/\tau$, the remaining laws are
\begin{equation}\label{eq:beta-generator-laws}
 \beta_0(g_jz)=\beta_0(z)+\varphi_j(z),\qquad j=1,2.
\end{equation}
The same theorem gives $R_{g_0}=I_2$. With $e_2=(0,1)^t$, the
fixed-column identity $A(g)\widehat\delta=\widehat\delta$ and the
equivariance at $a=0$ give
\[
 R_g(z)e_2
 =R_g(z)\Omega_0(z)\widehat\delta
 =\Omega_0(gz)A(g)\widehat\delta
 =e_2.
\]
Since $\Omega_a=\Omega_0+a e_2\gamma$ and
$\gamma A(g)^{-1}=\gamma$, it follows that
\[
 R_g(z)\Omega_a(z)A(g)^{-1}
 =\Omega_0(gz)+a e_2\gamma=\Omega_a(gz)
\]
for every $a\in\C$, with the same cocycle $R_g$. Together with the
representation law for $A$, this gives the orbifold descent action.
The real-lattice condition is imposed in
\cref{prop:admissible-half-plane}.
Finally, Theorem~3.4(iv) gives the invariance, continuity at the
finite orbifold points, and cusp limit of $\mathscr D_0$.  These
properties yield its upper bound on $B\setminus\{p_0\}$.
\end{proof}

\begin{proposition}
\label{prop:admissible-half-plane}
Let
\[
 M_0=\sup_{z\in U}\mathscr D_0(z)<\infty.
\]
Then
\begin{equation}
\label{eq:admissible}
 \mathcal A:=\{a\in\C:\Im a<-M_0\}
\end{equation}
is a nonempty open half-plane, and for every $a\in\mathcal A$ the four columns
of $\Omega_a(z)$ are linearly independent over $\R$ for all $z\in U$.
\end{proposition}

\begin{proof}
Equation \eqref{eq:D-a} gives
\[
 \mathscr D_a(z)=\mathscr D_0(z)+\Im a.
\]
If $\Im a<-M_0$, then $\mathscr D_a(z)<0$ for every $z$.

Write
\[
 Z_a=
 \begin{pmatrix}
 6\mu&\tau\\
 \beta_0+a&\mu
 \end{pmatrix},
 \qquad \Omega_a=[Z_a\mid I_2].
\]
A direct real block-determinant calculation gives
\begin{equation}
\label{eq:real-determinant}
 \det_{\R}\Omega_a=-\det(\Im Z_a)
 =\Im\tau\,\mathscr D_a.
\end{equation}
Here $\Im\tau>0$ and $\mathscr D_a<0$, so the determinant is nonzero.
\end{proof}

We call a parameter $a\in\mathcal A$ \emph{admissible}.

The period variation has one invariant lattice direction.
The invariant functional $\gamma$ has kernel
\[
 \Lambda_0=\ker\gamma
 =\Z\langle\widehat u,\widehat w,\widehat\delta\rangle.
\]
In particular, $\gamma A(g)=\gamma$ for every $g\in\Gamma$.
For any two admissible parameters,
\begin{equation}\label{nm:eq:invariant-direction}
 \Omega_a-\Omega_{a_*}=(a-a_*)e_2\gamma,\qquad
 \Omega_a\widehat\delta=e_2,\qquad R_ge_2=e_2.
\end{equation}
Translation by $u e_2$ therefore induces a compatible action
\[
 \Theta_{\exp(2\pi\ii u)}[z,\zeta_1,\zeta_2]
       =[z,\zeta_1,\zeta_2+u]
\]
on the quotients by $\Lambda_0$. Its extension across the finite and cusp
fillings will be checked from their formulas.

The equivariant period family determines the smooth core over the punctured
base. Fix a disc $\Delta\Subset\mathcal A$; the relative construction uses
parameters $a\in\Delta$.

Here $U$ is the orbifold universal cover, whereas
$U^\circ\to B^\circ$ is the ordinary covering
\eqref{eq:core-cover}, with deck group $\Gamma$. The space $U^\circ$
is not simply connected: a small loop around a deleted orbit point
has nonzero winding number about that point. The lattice quotient
below uses this marked covering and does not require its total space
to be simply connected.

Let $\Lambda$ act on $U^\circ\times\C^2\times\Delta$ by
\begin{equation}
\label{eq:lattice-action}
 \lambda\cdot(z,\zeta,a)
   =(z,\zeta+\Omega_a(z)\lambda,a).
\end{equation}

\begin{lemma}
\label{lem:relative-lattice-action}
The action \eqref{eq:lattice-action} is free and properly discontinuous.  Its
quotient
\[
 \widetilde{\mathcal J}_\Delta
   =(U^\circ\times\C^2\times\Delta)/\Lambda
\]
is a complex manifold, and the projection
\[
 \widetilde p:\widetilde{\mathcal J}_\Delta
   \longrightarrow U^\circ\times\Delta
\]
is a proper holomorphic submersion with fibres complex two-tori.
\end{lemma}

\begin{proof}
Freeness follows from the injectivity of the real-linear map
$\Omega_a(z):\Lambda_\R\to\C^2$, proved in
\cref{prop:admissible-half-plane}.

Let $K_1,K_2$ be compact subsets of
$U^\circ\times\C^2\times\Delta$.  Their projections to
$U^\circ\times\Delta$ lie in a compact set $K$.  The smallest singular value
of the real-linear map $\Omega_a(z)$ is a positive continuous function on
$K$, hence has a positive lower bound $c_K$.  If
$\lambda K_1\cap K_2\ne\varnothing$, then for suitable points with the same
$(z,a)$ the difference of their $\zeta$-coordinates is
$\Omega_a(z)\lambda$.  The compactness of $K_1,K_2$ bounds the norm of that
difference by a constant $C_K$.  Hence
\[
 c_K\|\lambda\|\le \|\Omega_a(z)\lambda\|\le C_K.
\]
Only finitely many lattice vectors satisfy this inequality.  Thus the action
is properly discontinuous, and Lee's holomorphic quotient theorem
\cite[Theorem~1.16]{LeeComplexManifolds} makes the quotient a complex
manifold.

For properness, set
\[
 \Phi:(U^\circ\times\Delta)\times(\Lambda_\R/\Lambda)
 \longrightarrow\widetilde{\mathcal J}_\Delta,
 \qquad
 ((z,a),[x])\longmapsto[z,\Omega_a(z)x,a].
\]
The map is well defined and bijective.  Its inverse is continuous because the
inverse real-linear matrices $\Omega_a(z)^{-1}$ depend continuously on
$(z,a)$.  Thus $\Phi$ is a homeomorphism over $U^\circ\times\Delta$.
Since $\Lambda_\R/\Lambda$ is compact, the projection is proper.  It is a
holomorphic submersion because locally it is the quotient of the coordinate
projection
$U^\circ\times\C^2\times\Delta\to U^\circ\times\Delta$ by translations.
\end{proof}

We next descend the lattice quotient by the deck group.

Equation \eqref{eq:period-equivariance} permits the action
\begin{equation}
\label{eq:Gamma-action}
 g\cdot[z,\zeta,a]=[gz,R_g(z)\zeta,a]
\end{equation}
on $\widetilde{\mathcal J}_\Delta$.  It is well defined because
\[
 R_g(z)\Omega_a(z)\lambda
   =\Omega_a(gz)A(g)\lambda.
\]
The group law follows from \eqref{eq:R-cocycle}.

\begin{proposition}
\label{prop:relative-core}
The action \eqref{eq:Gamma-action} is free and properly discontinuous, and the
quotient
\[
 \mathcal J_\Delta
   =\widetilde{\mathcal J}_\Delta/\Gamma
\]
is a connected complex manifold.  For every $a\in\Delta$, the parameter
slice $(\operatorname{pr}_\Delta\circ p^\circ)^{-1}(a)$ is connected.  The
induced map
\[
 p^\circ:\mathcal J_\Delta\longrightarrow B^\circ\times\Delta
\]
is a proper holomorphic submersion, and each fibre over a point of
$B^\circ\times\{a\}$ is a complex two-torus with local period matrix
$\Omega_a$.
\end{proposition}

\begin{proof}
The action of $\Gamma$ on $U^\circ$ is free and properly discontinuous.  If an
element fixes a point of $\widetilde{\mathcal J}_\Delta$, it fixes its base
point in $U^\circ$, hence is the identity.  For compact subsets of
$\widetilde{\mathcal J}_\Delta$, their projections to $U^\circ$ are compact;
only finitely many group elements can make the projected compact sets meet.
Thus the action is properly discontinuous.  Lee's holomorphic quotient
theorem \cite[Theorem~1.16]{LeeComplexManifolds} makes the quotient a complex
manifold.  The homeomorphism in
\cref{lem:relative-lattice-action} identifies
$\widetilde{\mathcal J}_\Delta$ topologically with
$(U^\circ\times\Delta)\times(\Lambda_\R/\Lambda)$; all three factors are
connected.  Hence $\widetilde{\mathcal J}_\Delta$ is connected, and its
continuous surjective quotient by $\Gamma$ is connected.  If $a\in\Delta$ is
fixed, the same homeomorphism identifies the parameter slice upstairs with
$U^\circ\times(\Lambda_\R/\Lambda)$, which is connected; its quotient by
$\Gamma$ is exactly $(\operatorname{pr}_\Delta\circ p^\circ)^{-1}(a)$ and is
therefore connected.

Let $V\subset B^\circ$ be a simply connected relatively compact coordinate
disc.  Choose one connected lift $\widetilde V\subset U^\circ$.  Every
$\Gamma$-orbit over $V$ meets $\widetilde V$ in exactly one point, because
$V$ is simply connected and $U^\circ\to B^\circ$ is a covering.  Hence the
restriction over $V\times\Delta$ is isomorphic to the lattice family over
$\widetilde V\times\Delta$, so the restriction of $p^\circ$ is proper by
\cref{lem:relative-lattice-action}. Properness is local on the target for
maps between Hausdorff spaces, and submersivity follows in the same local
chart.
\end{proof}

\subsection{Local fillings}
\label{sec:local-fillings}

The smooth family has three missing fibres, two at the finite orbifold points
and one at the cusp. Alp\"oge~\cite[Theorem~5.4]{AlpogeS6} constructs the
fixed-parameter logarithmic fillings at the finite points. We apply the
same affine actions to the family of period matrices $\Omega_a$ and keep the
parameter explicit in the quotient and gluing maps.

For $j=1,2$, set $m_1=3$ and $m_2=4$.  The marked base has the following
local uniformization data at $p_j$:
\begin{enumerate}[label=\textup{(F\arabic*)}]
\item
\begin{equation}
\label{eq:finite-stabilizer}
 \operatorname{Stab}_\Gamma(z_j)=\langle g_j\rangle,
 \qquad |\langle g_j\rangle|=m_j.
\end{equation}
\item There is a $g_j$-invariant disc
$\widetilde D_j\subset U$ centred at $z_j$ such that
\begin{equation}
\label{eq:finite-disc-separation}
 g\widetilde D_j\cap\widetilde D_j=\varnothing
 \qquad\text{for every }g\in\Gamma\setminus\langle g_j\rangle.
\end{equation}
The two discs are chosen so that their images in $B$ have disjoint closures.
\item There is a holomorphic coordinate $s_j$ on $\widetilde D_j$ satisfying
\begin{equation}
\label{eq:sj-action}
 s_j(g_jz)=\e^{-2\pi\ii/m_j}s_j(z).
\end{equation}
\end{enumerate}
The stabilizer assertion follows from the triangle uniformization above, and
proper discontinuity gives the discs in \textup{(F2)}.  The Cayley coordinate
$s_j(z)=(z-z_j)/(z-\overline z_j)$ conjugates $g_j$ to the rotation in
\eqref{eq:sj-action} under the chosen inverse-vertex convention.
Let
\[
 \widetilde D_j^*=\widetilde D_j\setminus\{z_j\}.
\]
By \eqref{eq:finite-disc-separation}, the quotient map identifies
\begin{equation}
\label{eq:finite-local-quotient}
 \widetilde D_j^*/\langle g_j\rangle
 \xrightarrow{\ \cong\ }D_j^*:=\varpi(\widetilde D_j^*)\subset B^\circ.
\end{equation}
Let $D_j$ be the image of $\widetilde D_j$ in the coarse quotient
$U/\Gamma\cong B\setminus\{p_0\}$.  Then $D_j=D_j^*\cup\{p_j\}$, and the
invariant coordinate $t_j=s_j^{m_j}$ identifies
$\widetilde D_j/\langle g_j\rangle$ biholomorphically with $D_j$.

The twist vectors are
\begin{equation}
\label{eq:twists}
 v_1=\widehat\gamma+2\widehat u-4\widehat w,
 \qquad
 v_2=-\widehat\gamma-3\widehat u+3\widehat w.
\end{equation}
Direct multiplication using \eqref{eq:A12} gives
\begin{equation}
\label{eq:twist-invariant}
 A_1v_1=v_1,\qquad A_2v_2=v_2.
\end{equation}
The invariant first-coordinate functional introduced above satisfies
\[
 \gamma(\widehat\gamma)=1,
 \qquad
 \gamma(\widehat u)=\gamma(\widehat w)
 =\gamma(\widehat\delta)=0.
\]
Then
\begin{equation}
\label{eq:twist-gamma}
 \gamma\circ A_j=\gamma,
 \qquad
 \gamma(v_1)=1,
 \qquad
 \gamma(v_2)=-1.
\end{equation}

Fix $j\in\{1,2\}$.  Let $\Lambda$ act on
$\widetilde D_j\times\C^2\times\Delta$ by
\[
 \lambda\cdot(z,\zeta,a)
   =(z,\zeta+\Omega_a(z)\lambda,a),
\]
and set
\[
 \mathcal P_{j,\Delta}
   =(\widetilde D_j\times\C^2\times\Delta)/\Lambda.
\]
The argument of \cref{lem:relative-lattice-action} applies on
$\widetilde D_j\times\Delta$: the action is free and properly discontinuous,
and Lee's holomorphic quotient theorem
\cite[Theorem~1.16]{LeeComplexManifolds} makes $\mathcal P_{j,\Delta}$ a
complex manifold.

The map
\[
 (\widetilde D_j\times\Delta)\times(\Lambda_\R/\Lambda)
 \longrightarrow \mathcal P_{j,\Delta},\qquad
 ((z,a),[x])\longmapsto[z,\Omega_a(z)x,a],
\]
is a homeomorphism over $\widetilde D_j\times\Delta$.  Since
$\Lambda_\R/\Lambda$ is compact, the projection
\[
 \mathcal P_{j,\Delta}\longrightarrow\widetilde D_j\times\Delta
\]
is proper.  In a local quotient chart it is induced by the coordinate
projection from $\widetilde D_j\times\C^2\times\Delta$, so its differential is
surjective and it is a holomorphic submersion.

The finite local data induce the following affine action.

Set
\begin{equation}
\label{eq:finite-affine-action}
 \widetilde g_j[z,\zeta,a]
  =\left[g_jz,
    R_{g_j}(z)\zeta+
    \Omega_a(g_jz)\frac{v_j}{m_j},a\right].
\end{equation}
This formula is jointly holomorphic in $(z,\zeta,a)$.  It is well defined on
the lattice quotient: replacing $\zeta$ by $\zeta+\Omega_a(z)\lambda$ changes
the target by
\[
 R_{g_j}(z)\Omega_a(z)\lambda
   =\Omega_a(g_jz)A_j\lambda,
\]
which is an integral target period.  The formula has a holomorphic inverse
because its $m_j$-th power is the identity, as proved next.

\begin{lemma}
\label{lem:finite-free}
The biholomorphism $\widetilde g_j$ has order exactly $m_j$, and the action of
$\langle\widetilde g_j\rangle$ on $\mathcal P_{j,\Delta}$ is free.
\end{lemma}

\begin{proof}
Use the real flat coordinate $x\in\Lambda_\R/\Lambda$ determined by
$\zeta=\Omega_a(z)x$.  Equation \eqref{eq:period-equivariance} transforms
\eqref{eq:finite-affine-action} into
\begin{equation}
\label{eq:flat-finite}
 (z,x,a)\longmapsto
 \left(g_jz,A_jx+\frac{v_j}{m_j},a\right).
\end{equation}
By \eqref{eq:twist-invariant}, $A_j^rv_j=v_j$ for every $r$.  Therefore the
$k$-th power has fibre translation
\[
 \frac{1}{m_j}\sum_{r=0}^{k-1}A_j^rv_j
   =\frac{k v_j}{m_j}.
\]
For $k=m_j$, the base rotation is the identity, $A_j^{m_j}=I$, and translation
by $v_j\in\Lambda$ is the identity on the torus.  Thus
$\widetilde g_j^{m_j}=1$.  For $1\le k<m_j$, the base rotation is nontrivial,
so the $k$-th power cannot be the identity; the order is exactly $m_j$.

A fixed point of a nonidentity power must lie over $z_j$, because the base
rotation has no other fixed point in $\widetilde D_j$.  At $z_j$, the fixed
point equation in the real torus is
\begin{equation}
\label{eq:finite-fixed-equation}
 (I-A_j^k)x\equiv \frac{k v_j}{m_j}\pmod\Lambda.
\end{equation}
Apply the integral functional $\gamma$.  Since
$\gamma\circ A_j=\gamma$, the left side has $\gamma$-value zero modulo
$\Z$.  The right side has value
\[
 \frac{k}{3}\quad(j=1),
 \qquad
 -\frac{k}{4}\quad(j=2),
\]
by \eqref{eq:twist-gamma}.  These numbers are not integers for
$1\le k<m_j$.  Hence \eqref{eq:finite-fixed-equation} has no solution.
\end{proof}

\begin{proposition}
\label{prop:relative-finite-filling}
The quotient
\[
 \mathcal N_{j,\Delta}
   =\mathcal P_{j,\Delta}/\langle\widetilde g_j\rangle
\]
is a connected complex manifold, and every parameter slice of
$\mathcal N_{j,\Delta}\to\Delta$ is connected.  It has a proper holomorphic
map
\[
 F_j:\mathcal N_{j,\Delta}\longrightarrow D_j\times\Delta,
\]
and its projection to $\Delta$ is a holomorphic submersion.  The fibre over
$p_j\times\{a\}$ has multiplicity $m_j$, and its reduced structure is a
compact smooth complex surface.
\end{proposition}

\begin{proof}
The group is finite and acts freely, so Lee's holomorphic quotient theorem
\cite[Theorem~1.16]{LeeComplexManifolds} applies.
The flat-coordinate homeomorphism above identifies
$\mathcal P_{j,\Delta}$ with the product of the connected space
$\widetilde D_j\times\Delta$ and the connected real torus
$\Lambda_\R/\Lambda$.  Thus $\mathcal P_{j,\Delta}$ is connected, and its
continuous surjective finite quotient $\mathcal N_{j,\Delta}$ is connected.
After fixing $a$, the same product description gives the connected space
$\widetilde D_j\times(\Lambda_\R/\Lambda)$ upstairs; its finite quotient is
the parameter slice of $\mathcal N_{j,\Delta}$ and is connected.
Let
\[
 \rho_j:\widetilde D_j\times\Delta\longrightarrow D_j\times\Delta,
 \qquad (z,a)\longmapsto(s_j(z)^{m_j},a).
\]
The composite
\[
 \mathcal P_{j,\Delta}\longrightarrow
 \widetilde D_j\times\Delta
 \xrightarrow{\ \rho_j\ }D_j\times\Delta
\]
is invariant under $\widetilde g_j$, because
$s_j(g_jz)^{m_j}=s_j(z)^{m_j}$.  It therefore descends to the holomorphic
map $F_j$ in the statement.

The map $\rho_j$ is finite and hence proper.  If
$K\subset D_j\times\Delta$ is compact, its inverse image in
$\mathcal P_{j,\Delta}$ is compact because the torus family is proper.  The
image of this compact set in the finite quotient is exactly $F_j^{-1}(K)$.
Thus $F_j$ is proper.

The parameter projection upstairs is a submersion and is invariant under the
finite group.  Since the quotient map is locally biholomorphic, the descended
parameter projection is again a submersion.

Finally, the coarse base coordinate is $t_j=s_j^{m_j}$.  Its pullback upstairs
vanishes to order $m_j$ along the central torus.  The quotient map of the total
space is locally biholomorphic because the action is free.  Upstairs, the
reduced central divisor is the smooth torus family $s_j=0$; its quotient by
the free finite action is therefore smooth and is the reduction of
$F_j^{-1}(p_j\times\Delta)$.  Consequently this Cartier divisor has
multiplicity $m_j$ along its smooth reduction.
\end{proof}

On a simply connected sector of $\widetilde D_j^*$, choose a branch of
$\log s_j$ and set
\begin{equation}
\label{eq:sigma-ja}
 \sigma_{j,a}(z)=
 \frac{\log s_j(z)}{2\pi\ii}\,\Omega_a(z)v_j.
\end{equation}
The fibre translation is
\[
 h_j[z,\zeta,a]=[z,\zeta+\sigma_{j,a}(z),a].
\]
If the branch of the logarithm changes by $2\pi\ii n$, then
$\sigma_{j,a}$ changes by $n\Omega_a(z)v_j$, an integral period.  Hence the
translation on the torus quotient is independent of the branch.

\begin{lemma}
\label{lem:finite-conjugation}
On $\mathcal P_{j,\Delta}|_{\widetilde D_j^*\times\Delta}$ one has
\[
 h_j\widetilde g_jh_j^{-1}=\widetilde g_j^0,
 \qquad
 \widetilde g_j^0[z,\zeta,a]=[g_jz,R_{g_j}(z)\zeta,a].
\]
Consequently there is a jointly holomorphic biholomorphism
\begin{equation}
\label{eq:Gj}
 G_j:\mathcal N_{j,\Delta}|_{D_j^*\times\Delta}
 \xrightarrow{\ \cong\ }
 \mathcal J_\Delta|_{D_j^*\times\Delta}.
\end{equation}
\end{lemma}

\begin{proof}
From \eqref{eq:sj-action}, on a chosen sector one has
\[
 \log s_j(g_jz)=\log s_j(z)-\frac{2\pi\ii}{m_j}.
\]
Using \eqref{eq:period-equivariance} and $A_jv_j=v_j$, we get
\[
 R_{g_j}(z)\Omega_a(z)v_j
  =\Omega_a(g_jz)A_jv_j
  =\Omega_a(g_jz)v_j.
\]
Therefore
\begin{align*}
 \sigma_{j,a}(g_jz)
 &=\frac{\log s_j(g_jz)}{2\pi\ii}\,
   \Omega_a(g_jz)v_j\\
 &=R_{g_j}(z)\sigma_{j,a}(z)
   -\Omega_a(g_jz)\frac{v_j}{m_j}.
\end{align*}
Substituting this identity into the definitions of $h_j$ and
$\widetilde g_j$ gives the displayed conjugation formula term by term.

We now identify the untwisted quotient with the actual restriction of the
global core.  Let
\[
 \widetilde{\mathcal J}_{j,\Delta}^*
 =
 \widetilde{\mathcal J}_\Delta
 \big|_{\widetilde D_j^*\times\Delta}.
\]
The subgroup $\langle g_j\rangle$ acts on this space by the untwisted maps
$\widetilde g_j^0$.  If
$g\widetilde{\mathcal J}_{j,\Delta}^*
 \cap\widetilde{\mathcal J}_{j,\Delta}^*\ne\varnothing$,
then the projections to $U^\circ$ give
$g\widetilde D_j^*\cap\widetilde D_j^*\ne\varnothing$.
Equation \eqref{eq:finite-disc-separation} therefore implies
$g\in\langle g_j\rangle$, so
$\widetilde{\mathcal J}_{j,\Delta}^*$ is precisely invariant under
$\langle g_j\rangle$. The orbit maps are open, and Lee's local quotient
charts \cite[Theorem~1.16]{LeeComplexManifolds} give a biholomorphism
\[
 \widetilde{\mathcal J}_{j,\Delta}^*/\langle g_j\rangle
 \xrightarrow{\ \cong\ }
 \mathcal J_\Delta|_{D_j^*\times\Delta}.
\]
The affine quotient of the same pulled-back family is
$\mathcal N_{j,\Delta}|_{D_j^*\times\Delta}$.  Hence the conjugation descends
to \eqref{eq:Gj}.  Every formula is holomorphic in $(z,a)$, and branch changes
alter the translation by an integral period, so the descended map is jointly
holomorphic and single-valued.
\end{proof}

The two finite-point fillings are now identified with the smooth core on
their punctured discs. At the remaining point, the cusp quotient requires a
Hausdorff countable-fan model and a proper lattice action that remains
controlled as $q\to0$. It must also agree with the smooth core over the
punctured cusp disc.

We choose a simply connected cusp neighbourhood
$\mathcal H_0\subset U^\circ$ and a biholomorphic coordinate
\[
 s:\mathcal H_0\xrightarrow{\ \cong\ }
 \{w\in\C:\Im w>R_0\}
\]
for some real number $R_0$, satisfying
\[
 s(g_0z)=s(z)-1.
\]
The neighbourhood and coordinate can be chosen to obey
\begin{align}
 \operatorname{Stab}_\Gamma(\mathcal H_0)
   &=\langle g_0\rangle,\label{eq:cusp-stabilizer}\\
 g\mathcal H_0\cap\mathcal H_0
   &=\varnothing
   \qquad(g\in\Gamma\setminus\langle g_0\rangle).
   \label{eq:cusp-horodisc-separation}
\end{align}
Indeed, a sufficiently small horodisc at the parabolic fixed point
of $g_0$ is precisely invariant under $\langle g_0\rangle$ by the
discrete triangle-group action fixed above.
The function $s=\tau-h(t_c)$ constructed in
\cref{prop:period-data-existence} satisfies
$\e^{2\pi\ii s}=t_c$. On a smaller cusp disc, $h$ is bounded, so a
high enough half-plane in the $s$ coordinate lies inside that precisely
invariant horodisc. Its inverse image is $\mathcal H_0$, proving
\eqref{eq:cusp-stabilizer} and \eqref{eq:cusp-horodisc-separation}.
Here the first equality means that the setwise stabilizer of the cusp
neighbourhood is the infinite cyclic group generated by the primitive
parabolic element $g_0$. After increasing $R_0$ and replacing
$\mathcal H_0$ by the inverse image of the corresponding higher half-plane if
necessary, set
\[
 q=\e^{2\pi\ii s},
 \qquad
 \varepsilon_0=\e^{-2\pi R_0}.
\]
The function $q$ identifies the quotient
$\mathcal H_0/\langle g_0\rangle$ biholomorphically with the punctured disc
\begin{equation}
\label{eq:cusp-local-quotient}
 D_0^*=\{q\in\C:0<|q|<\varepsilon_0\}.
\end{equation}
Equation \eqref{eq:cusp-horodisc-separation} makes $\mathcal H_0$ precisely
invariant under $\langle g_0\rangle$, so the global orbit map embeds this
quotient biholomorphically in $B^\circ$; its image is identified with
$D_0^*$.  Let
\[
 D_0=D_0^*\cup\{p_0\}\cong\{|q|<\varepsilon_0\}.
\]
The three discs $D_0,D_1,D_2$ are chosen with pairwise disjoint closures.

In the splitting
\[
 \Lambda_{\tor}=\Z\langle\widehat w,\widehat\delta\rangle,
 \qquad
 \overline\Lambda=\Lambda/\Lambda_{\tor}
   =\Z\langle\widehat\gamma,\widehat u\rangle,
\]
write $\Omega_a=[Z_a\mid I_2]$. The nonconstant
$2\times2$ period block has the form
\begin{equation}
\label{eq:cusp-decomposition}
 Z_a=sB_0+C_a(q),
 \qquad
 B_0=
 \begin{pmatrix}0&1\\-1&0\end{pmatrix},
 \qquad
 C_a(q)=C_0(q)+aE_{21},
\end{equation}
where $C_0$ is holomorphic on a disc $D_{\varepsilon_0}$ about $q=0$ and
$E_{21}$ has a single $1$ in row $2$, column $1$.
To verify the decomposition, let $b=\beta_0+\tau$. The proof of
\cref{prop:period-data-existence} shows that $b$ and $\mu$ are
$g_0$-invariant and holomorphic in $q=t_c$ at $0$, and that
$\tau=s+h(q)$. Hence the displayed period matrix gives
\[
 C_0(q)=\begin{pmatrix}6\mu(q)&h(q)\\
                       b(q)-h(q)&\mu(q)\end{pmatrix},
\]
which is holomorphic at $q=0$. The same proposition computes
$R_{g_0}=I_2$, and the additive parameter contributes $aE_{21}$.
Since $q(g_0z)=q(z)$ and
$s(g_0z)=s(z)-1$, one has
\begin{equation}
\label{eq:cusp-deck-period}
 Z_a(g_0z)=Z_a(z)-B_0.
\end{equation}
These identities will be used in the punctured identification.

We recall the toric notation used for the periodic fan.

For a lattice $N$, write $M=\Hom(N,\Z)$ and
$N_\R=N\otimes_\Z\R$.  A \emph{rational polyhedral cone} is a strongly
convex cone generated by finitely many elements of $N$.  A
\emph{rational fan} $\Sigma$ is a collection of rational polyhedral cones
such that every face of a cone in $\Sigma$ belongs to $\Sigma$ and the
intersection of two cones in $\Sigma$ is a face of each.  No finiteness or
ambient local-finiteness condition is included in this definition.  For
$\sigma\in\Sigma$, set
\[
   U_\sigma=\Spec\C[\sigma^\vee\cap M]
\]
and use the same notation for its complex analytification.

For finite subfans we use the standard face-chart embeddings,
separated gluing, orbit--cone correspondence, and star-fan description of
toric divisors from Fulton~\cite[Sections~1.3--1.4, 2.1, 3.1 and~3.3]{Fulton}.
For the specific periodic fan below, the countable gluing and its torus
action are established by Alp\"oge~\cite[Lemmas~4.2--4.3]{AlpogeS6}.
We identify that fan with the present coordinates next.

Let $N=\Z^2$ with standard basis $e_1,e_2$, and triangulate $N_\R=\R^2$ by
the triangles
\begin{align*}
 P_v^+&=\operatorname{conv}\{v,v+e_1,v+e_2\},\\
 P_v^-&=\operatorname{conv}\{v+e_1,v+e_2,v+e_1+e_2\},
 \qquad v\in\Z^2.
\end{align*}
Let $N'=N\oplus\Z$.  For every vertex, edge, or triangle $P$ of the
triangulation, set
\[
 \sigma_P=\operatorname{cone}(P\times\{1\})\subset N'_\R,
\]
and set
\[
 \Sigma=\{0\}\cup\{\sigma_P:P\text{ is a cell of the triangulation}\}.
\]

\begin{proposition}
\label{prop:periodic-conical-fan}
The collection $\Sigma$ is a countable rational fan.  Every cone is generated
by part of a basis of $N'$.  For each vertex $v\in\Z^2$, the star fan of the
ray $\R_{\ge0}(v,1)$ is the finite complete hexagonal fan.  Consequently the
countable toric space
\[
 Y=Y_\Sigma
\]
is a smooth Hausdorff second-countable complex threefold, and the toric
divisor $E_v$ corresponding to the ray through $(v,1)$ is compact.
\end{proposition}

\begin{proof}
The two families of triangles above are exactly the periodic
$A_2$-triangulation used by Alp\"oge~\cite[Lemma~4.2(i)--(iv)]{AlpogeS6},
under the identification of its height-one lattice with
$N'=\Z^2\oplus\Z$.  In particular, that lemma proves the fan axioms,
unimodularity, Hausdorff and second-countable toric gluing, and the complete
hexagonal star fan of each ray through $(v,1)$.  The source's toric
threefold and divisors are therefore precisely $Y_\Sigma$ and $E_v$ in the
present notation.
\end{proof}

Let $h=(0,0,1)\in M'=\Hom(N',\Z)$.  Since $h$ is nonnegative on every cone,
the character $\chi^h$ is holomorphic on every affine toric chart and the
chartwise functions glue to
\[
 q:Y\longrightarrow\C.
\]
If a maximal cone is generated by the rays through
$(v_0,1),(v_1,1),(v_2,1)$ and $m_0,m_1,m_2$ is the dual basis, then
$h=m_0+m_1+m_2$.  In the affine coordinates $z_i=\chi^{m_i}$ one has
\begin{equation}
\label{eq:q-local}
 q=z_0z_1z_2.
\end{equation}
The divisor $q^{-1}(0)$ is therefore reduced with simple normal crossings.
The component $E_v$ corresponding to the ray through $(v,1)$ has the compact
hexagonal star fan described in
\cref{prop:periodic-conical-fan}.  We shall use repeatedly that $E_0$ is a
compact smooth toric surface.

All matrix norms in the definition below and in the cusp analysis are the
operator norms induced by the standard Euclidean norms on the relevant real
vector spaces.

\begin{definition}
\label{def:valid-cusp-cutoff}
Let $\Delta\Subset\mathcal A$ be a relatively compact parameter disc.  A
number $\varepsilon>0$ is called a \emph{valid cusp cutoff for $\Delta$} if
there exists a number $\varepsilon_1$ such that
\begin{enumerate}[label=\textup{(V\arabic*)}]
\item
\[
  0<\varepsilon<\varepsilon_1<\min\{1,\varepsilon_0\};
\]
\item with
\[
 M_C(\Delta,\varepsilon_1)
 =\sup_{\substack{|q|\le\varepsilon_1\\a\in\overline\Delta}}
   \bigl\|-2\pi\Im C_a(q)\bigr\|,
\]
one has
\begin{equation}
\label{eq:valid-cutoff-inequality}
 \frac{\|B_0^{-1}\|M_C(\Delta,\varepsilon_1)}
      {|\log\varepsilon|}\le\frac12;
\end{equation}
\item the closed subdisc $\overline D_\varepsilon=\{|q|\le\varepsilon\}$,
      viewed inside the original cusp coordinate disc
      $D_{\varepsilon_0}\cup\{p_0\}\subset B$, is disjoint from
      $\overline D_1\cup\overline D_2$.
\end{enumerate}
A number $\varepsilon_1$ satisfying these conditions is called an
\emph{auxiliary cutoff} for $\varepsilon$.
\end{definition}

\begin{lemma}
\label{lem:valid-cutoff-exists}
For every relatively compact parameter disc
$\Delta\Subset\mathcal A$, valid cusp cutoffs exist.  If $\varepsilon$ is
valid for $\Delta$ and $0<\varepsilon'<\varepsilon$, then
$\varepsilon'$ is valid for $\Delta$ with the same auxiliary cutoff.
\end{lemma}

\begin{proof}
Choose any
$0<\varepsilon_1<\min\{1,\varepsilon_0\}$.  Since $C_0$ is holomorphic on
$D_{\varepsilon_0}$, the function
$(q,a)\mapsto C_a(q)=C_0(q)+aE_{21}$ is continuous on the compact set
$\overline D_{\varepsilon_1}\times\overline\Delta$.  Thus
$M_C(\Delta,\varepsilon_1)<\infty$.  As
$\varepsilon\downarrow0$, the denominator $|\log\varepsilon|$ tends to
infinity, so condition~\textup{(V2)} holds for all sufficiently small
positive $\varepsilon<\varepsilon_1$.  The original cusp disc was chosen with
closure disjoint from the two finite discs; hence the same
is true for every sufficiently small closed subdisc.  This proves existence.

If $0<\varepsilon'<\varepsilon$, then
$|\log\varepsilon'|>|\log\varepsilon|$, so the left side of
\eqref{eq:valid-cutoff-inequality} can only decrease.  The smaller closed
cusp disc is contained in the larger one, so condition~\textup{(V3)} is also
preserved.  The same $\varepsilon_1$ is therefore an auxiliary cutoff for
$\varepsilon'$.
\end{proof}

The cusp is the only place where an infinite group acts on a non-quasi-compact
toric space.  We give separate proofs of joint holomorphic extension,
freeness, proper discontinuity, and properness of the quotient map.
The underlying fixed-parameter cusp filling, its free and proper lattice
quotient, the description of its central fibre, and its punctured gluing are
due to Alp\"oge~\cite[Theorem~4.5 and Propositions~4.6--4.7]{AlpogeS6}.
Here the matrices $C_a$ vary holomorphically with $a$, and the estimates below
are made uniform on the parameter disc $\Delta$.

Fix a valid cusp cutoff $\varepsilon$ for $\Delta$ in the sense of
Definition~\ref{def:valid-cusp-cutoff}, and choose one auxiliary cutoff
$\varepsilon_1$.  Set
\begin{equation}
\label{eq:MC-bound}
 M_C=M_C(\Delta,\varepsilon_1)
 =\sup_{\substack{|q|\le\varepsilon_1\\a\in\overline\Delta}}
 \bigl\|-2\pi\Im C_a(q)\bigr\|<\infty.
\end{equation}
By condition~\textup{(V2)} in
Definition~\ref{def:valid-cusp-cutoff},
\begin{equation}
\label{eq:epsilon-choice}
 \frac{\|B_0^{-1}\|M_C}{|\log\varepsilon|}\le\frac12.
\end{equation}
We now restrict every cusp object to this radius and write
\[
 D_0=D_\varepsilon=\{q\in\C:|q|<\varepsilon\},
 \qquad
 D_0^*=D_\varepsilon\setminus\{0\},
 \qquad
 Y_\varepsilon=q^{-1}(D_\varepsilon).
\]
On the uniformizing side take the smaller horodisc
\begin{equation}
\label{eq:smaller-horodisc}
 \mathcal H_{0,\varepsilon}
   =\{z\in\mathcal H_0:0<|q(z)|<\varepsilon\}.
\end{equation}
Since $|q(z)|=\exp(-2\pi\Im s(z))$, this is a connected horodisc.  It is
$\langle g_0\rangle$-invariant, and the quotient map $q$ induces a
biholomorphism
\begin{equation}
\label{eq:smaller-cusp-quotient}
 \mathcal H_{0,\varepsilon}/\langle g_0\rangle
 \xrightarrow{\ \cong\ }D_0^*.
\end{equation}
Moreover, if
$g\mathcal H_{0,\varepsilon}\cap\mathcal H_{0,\varepsilon}\ne\varnothing$,
then
$g\mathcal H_0\cap\mathcal H_0\ne\varnothing$; by
\eqref{eq:cusp-horodisc-separation}, this forces
$g\in\langle g_0\rangle$.  Thus the isolated-subgroup property survives the
shrinking.  Condition~\textup{(V3)} in
Definition~\ref{def:valid-cusp-cutoff} gives
$\overline D_0\cap(\overline D_1\cup\overline D_2)=\varnothing$.

Identify $\overline\Lambda$ with $\Z^2$ in the basis
$(\widehat\gamma,\widehat u)$.  For $\lambda\in\Z^2$, the lattice automorphism
\begin{equation}
\label{eq:shear}
 \phi_\lambda:N'\longrightarrow N',
 \qquad
 (y,r)\longmapsto(y+rB_0\lambda,r)
\end{equation}
preserves the fan $\Sigma$.  Let $\Phi_\lambda$ be the induced toric
automorphism of $Y$.  On the dense torus $(\C^*)^2\times\C^*$ it is
\[
 \Phi_\lambda(x,q)=(q^{B_0\lambda}x,q).
\]

For $(q,a)\in D_\varepsilon\times\Delta$, set
\begin{equation}
\label{eq:c-lambda}
 c_{a,\lambda}(q)
  =\exp(2\pi\ii C_a(q)\lambda)\in(\C^*)^2,
\end{equation}
where exponentiation is componentwise.

\begin{proposition}
\label{prop:joint-cusp-action}
For every $\lambda\in\Z^2$, the map
\begin{equation}
\label{eq:Psi-action}
 \Psi_\lambda(y,a)
  =\bigl(c_{a,\lambda}(q(y))\cdot\Phi_\lambda(y),a\bigr)
\end{equation}
is a biholomorphism of $Y_\varepsilon\times\Delta$. The dependence on
$(y,a)$ is holomorphic, and
\[
 \Psi_\lambda\Psi_\mu=\Psi_{\lambda+\mu}.
\]
On the dense torus,
\begin{equation}
\label{eq:Psi-dense}
 \Psi_\lambda(x,q,a)
 =\bigl(c_{a,\lambda}(q)q^{B_0\lambda}x,q,a\bigr).
\end{equation}
\end{proposition}

\begin{proof}
Alp\"oge~\cite[Lemma~4.2(i)]{AlpogeS6} gives the holomorphic
torus multiplication map $(\C^*)^2\times Y\to Y$ for this fan.  The map
\[
 (y,a)\longmapsto c_{a,\lambda}(q(y))
\]
is holomorphic because $C_a(q)=C_0(q)+aE_{21}$ is jointly holomorphic in
$(q,a)$.  Therefore torus translation by this point is jointly holomorphic on
$Y\times\Delta$.  Composing it with the toric biholomorphism $\Phi_\lambda$
gives \eqref{eq:Psi-action}.

Both the vertical torus action and $\Phi_\lambda$ preserve $q$.  Hence
$\Psi_\lambda$ preserves $Y_\varepsilon\times\Delta$.

The shear \eqref{eq:shear} restricts to the identity on the vertical sublattice
$N\oplus\{0\}$.  Hence $\Phi_\lambda$ commutes with the vertical torus action.
Moreover,
\[
 c_{a,\lambda}(q)c_{a,\mu}(q)=c_{a,\lambda+\mu}(q)
\]
because $C_a(q)$ is linear in the lattice vector.  The shears add:
$\phi_\lambda\phi_\mu=\phi_{\lambda+\mu}$.  These identities prove the group
law.  Taking $\mu=-\lambda$ shows that $\Psi_{-\lambda}$ is the inverse of
$\Psi_\lambda$.  Formula \eqref{eq:Psi-dense} is the corresponding expression
on the dense torus.
\end{proof}

Thus the action extends jointly holomorphically across the toric boundary.

For $0<|q|<\varepsilon$, set
\begin{equation}
\label{eq:R-and-Bqa}
 R_a(q)=-2\pi\Im C_a(q),
 \qquad
 B_{q,a}=B_0+\frac{R_a(q)}{\log|q|}.
\end{equation}
Here $R_a(q)$ is a real $2\times2$ matrix.

\begin{lemma}
\label{lem:uniform-invertibility}
For the radius $\varepsilon$ chosen in \eqref{eq:epsilon-choice}, there is a
constant $K>0$ such that
\[
 \|B_{q,a}\|+\|B_{q,a}^{-1}\|\le K
\]
for every $0<|q|<\varepsilon$ and every $a\in\overline\Delta$.
\end{lemma}

\begin{proof}
Set
\[
 E_{q,a}=B_0^{-1}\frac{R_a(q)}{\log|q|}.
\]
By \eqref{eq:MC-bound} and \eqref{eq:epsilon-choice},
\[
 \|E_{q,a}\|\le
 \frac{\|B_0^{-1}\|M_C}{|\log\varepsilon|}\le\frac12,
 \qquad B_{q,a}=B_0(I+E_{q,a}).
\]
The standard Neumann estimate gives $\|(I+E_{q,a})^{-1}\|\le2$.
Consequently,
\[
 \|B_{q,a}\|
 \le \|B_0\|+\frac{M_C}{|\log\varepsilon|},
 \qquad
 \|B_{q,a}^{-1}\|\le2\|B_0^{-1}\|.
\]
Their sum is bounded independently of $q$ and $a$.
\end{proof}

On the dense torus use the logarithmic modulus coordinate
\begin{equation}
\label{eq:y-coordinate}
 y(x,q)=\frac{\log|x|}{\log|q|}\in\R^2.
\end{equation}
Taking absolute values in \eqref{eq:Psi-dense} gives
\begin{equation}
\label{eq:y-translation}
 y(\Psi_\lambda(x,q,a))=y(x,q)+B_{q,a}\lambda.
\end{equation}

The next lemma is the point that controls mixed sequences with $q\ne0$ and
$q\to0$.

\begin{lemma}
\label{lem:log-bound}
Let $K\subset Y_\varepsilon$ be compact.  There exists $M_K>0$ such that
\[
 \|y(x,q)\|\le M_K
\]
for every $(x,q)\in K$ lying in the dense torus with $q\ne0$.
\end{lemma}

\begin{proof}
For every point of $K$, choose a relatively compact open neighbourhood whose
closure lies in an affine chart associated with a maximal triangle.  A finite
number of these smaller neighbourhoods covers $K$.  Their intersections with
$K$ have compact closures inside the corresponding affine charts.  It is
therefore enough to prove the assertion on one such compact chart piece.

Write $P=\operatorname{conv}\{v_0,v_1,v_2\}$ for the corresponding maximal
triangle. Its cone is generated by $(v_i,1)$, which form a basis of $N'$.
For the dual basis $m_0,m_1,m_2$, write $z_i=\chi^{m_i}$. The affine
barycentric functions are
\[
 \ell_i(y)=\langle m_i,(y,1)\rangle.
\]
They satisfy
\[
 \ell_i(v_j)=\delta_{ij},
 \qquad
 \ell_0+\ell_1+\ell_2=1,
 \qquad
 y=\sum_{i=0}^2\ell_i(y)v_i.
\]
For a dense-torus point, the character formula gives
\begin{equation}
\label{eq:zi-barycentric}
 \log|z_i|=\ell_i(y)\log|q|.
\end{equation}

On the chosen compact chart piece, choose $R\ge1$ such that
$|z_i|\le R$ for all $i$, and set
\[
 c=\frac{\log R}{|\log\varepsilon|}.
\]
For every $0<|q|<\varepsilon<1$, Equation \eqref{eq:zi-barycentric} gives
\[
 \ell_i(y)\ge-c,\qquad
 \ell_i(y)=1-\sum_{j\ne i}\ell_j(y)\le1+2c.
\]
Consequently
\[
 \|y\|\le(1+2c)\sum_{i=0}^2\|v_i\|.
\]
Taking the maximum over the finite chart cover proves the lemma.
\end{proof}

\begin{proposition}
\label{prop:cusp-free}
The $\Z^2$-action \eqref{eq:Psi-action} on
$Y_\varepsilon\times\Delta$ is free.
\end{proposition}

\begin{proof}
Suppose first that $q\ne0$ and $\Psi_\lambda(x,q,a)=(x,q,a)$.  Equation
\eqref{eq:y-translation} gives $B_{q,a}\lambda=0$.  By
\cref{lem:uniform-invertibility}, $\lambda=0$.

Now let $q=0$.  Torus orbits in the central fibre correspond to cells $P$ of
the height-one triangulation.  The toric automorphism $\Phi_\lambda$ sends the
orbit of $P$ to the orbit of $P+B_0\lambda$.  The additional torus translation
$c_{a,\lambda}(0)$ preserves each torus orbit.  Therefore a fixed point would
require
\[
 P+B_0\lambda=P.
\]
Every cell $P$ is bounded, so it is invariant under a translation only when
the translation vector is zero.  Since $B_0$ is invertible, $\lambda=0$.
\end{proof}

\begin{proposition}
\label{prop:cusp-proper-discontinuous}
The $\Z^2$-action on $Y_\varepsilon\times\Delta$ is properly discontinuous.
\end{proposition}

\begin{proof}
Let $K_1,K_2\subset Y_\varepsilon\times\Delta$ be compact.  We must show that
only finitely many $\lambda\in\Z^2$ satisfy
\[
 \Psi_\lambda(K_1)\cap K_2\ne\varnothing.
\]

First consider intersections represented by points with $q\ne0$.  The
projections of $K_1,K_2$ to $Y_\varepsilon$ are compact, so
\cref{lem:log-bound} gives constants $M_1,M_2$ bounding the logarithmic
coordinates of the source and target points.  If
$p\in K_1$ and $\Psi_\lambda(p)\in K_2$, then
\eqref{eq:y-translation} gives
\[
 \lambda=B_{q,a}^{-1}(y(\Psi_\lambda p)-y(p)).
\]
By \cref{lem:uniform-invertibility},
\[
 \|\lambda\|\le K(M_1+M_2).
\]
Only finitely many integer vectors satisfy this bound.  Notice that this
argument remains valid for sequences $q\ne0$ tending to $0$; that is exactly
the content of \cref{lem:log-bound}.

It remains to consider intersections lying in $q=0$.  For $i=1,2$, let
\[
 K_i^0=\operatorname{pr}_{Y_\varepsilon}
 \bigl(K_i\cap(q^{-1}(0)\times\Delta)\bigr).
\]
Each $K_i^0$ is compact.  A finite collection of affine toric charts covers
each $K_i^0$, and every chosen chart contains only finitely many torus orbits,
corresponding to the faces of its cone.  Hence there are finite sets of cells
$\mathscr S_1,\mathscr S_2$ such that $K_i^0$ meets only the orbits associated
with cells in $\mathscr S_i$.  If an orbit meeting $K_1^0$ is sent to one
meeting $K_2^0$, then
\[
 P+B_0\lambda=P'
\]
for some $P\in\mathscr S_1$ and $P'\in\mathscr S_2$.  For each pair $(P,P')$
there is at most one $\lambda$, because $B_0$ is invertible.  Thus only
finitely many $\lambda$ occur on the central fibre as well.
\end{proof}

\begin{lemma}
\label{lem:bounded-region-compact}
If $K_y\subset\R^2$ is bounded and $0<r<\varepsilon$, then the closure in
$Y$ of
\[
 \mathcal S(K_y,r)=
 \{(x,q)\in(\C^*)^2\times\C^*:
       |q|\le r,\ y(x,q)\in K_y\}
\]
is compact.
\end{lemma}

\begin{proof}
It suffices to replace $K_y$ by $\overline{K_y}$, since
$\mathcal S(K_y,r)\subset\mathcal S(\overline{K_y},r)$.  We may therefore
assume that $K_y$ is compact.  It meets only finitely many triangles of the
locally finite triangulation.  For each $y\in K_y$, choose a triangle
$P=\operatorname{conv}\{v_0,v_1,v_2\}$ containing $y$.  In the corresponding
affine chart, the barycentric functions satisfy $\ell_i(y)\ge0$ and
$\sum_i\ell_i(y)=1$.  Equation \eqref{eq:zi-barycentric} gives
\[
 |z_i|=|q|^{\ell_i(y)}\le1.
\]
Thus the part of $\mathcal S(K_y,r)$ associated with $P$ lies in the closed
unit polydisc of the chart $U_P\cong\C^3$.  Only finitely many triangles are
needed, so the closure is contained in a finite union of compact closed
polydiscs.  It is closed in that finite union and is therefore compact.
\end{proof}

\begin{proposition}
\label{prop:compact-fundamental}
Let $K\subset D_\varepsilon\times\Delta$ be compact.  There exists a compact
set
\[
 C_K\subset (q,a)^{-1}(K)\subset Y_\varepsilon\times\Delta
\]
meeting every $\Z^2$-orbit over $K$.
\end{proposition}

\begin{proof}
Choose $r<\varepsilon$ and a compact set $K_a\Subset\Delta$ such that
$K\subset\overline D_r\times K_a$.  For $0<|q|\le r$ and $a\in K_a$, set
\[
 \mathcal P_{q,a}=B_{q,a}([-1/2,1/2]^2).
\]
By \cref{lem:uniform-invertibility}, all these parallelograms lie in one
bounded set $K_y\subset\R^2$.

Let $(x,q,a)$ have $q\ne0$.  Choose $\lambda\in\Z^2$ such that
\[
 B_{q,a}^{-1}y(x,q)+\lambda\in[-1/2,1/2]^2.
\]
Then \eqref{eq:y-translation} gives
\[
 y(\Psi_\lambda(x,q,a))
 =B_{q,a}\bigl(B_{q,a}^{-1}y(x,q)+\lambda\bigr)
 \in K_y.
\]
Therefore the compact set
\[
 C^*=\overline{\mathcal S(K_y,r)}\times K_a
\]
meets every orbit with $q\ne0$ over the larger box
$\overline D_r\times K_a$, by
\cref{lem:bounded-region-compact}.

Now consider $q=0$.  Let a point lie in the orbit corresponding to a cell
$P$.  Choose a vertex $v$ of $P$.  Since $B_0\in\GL(2,\Z)$, there is a unique
$\lambda\in\Z^2$ with $B_0\lambda=-v$.  The translated cell
$P+B_0\lambda$ contains $0$.  Hence the translated torus orbit is contained
in the component $E_0$ of the central fibre.  The torus factor
$c_{a,\lambda}(0)$ preserves that orbit and therefore does not alter this
conclusion.  Thus the compact set $E_0\times K_a$ meets every central-fibre
orbit.

Set
\[
 C_{\mathrm{box}}=C^*\cup(E_0\times K_a).
\]
It is compact and meets every orbit over $\overline D_r\times K_a$.  Because
the group action preserves $(q,a)$, the set
\[
 C_K=C_{\mathrm{box}}\cap(q,a)^{-1}(K)
\]
is compact and still meets every orbit over $K$.
\end{proof}

\begin{lemma}
\label{lem:toric-central-normalization}
Let
\[
 D_Y=q^{-1}(0)\subset Y.
\]
Then $D_Y$ is reduced, and its normalization is the disjoint union
\begin{equation}
\label{eq:DY-normalization}
 \nu_Y:\coprod_{v\in\Z^2}E_v\longrightarrow D_Y
\end{equation}
of its smooth irreducible toric components.

\end{lemma}

\begin{proof}
In a maximal toric chart the central equation is $q=z_0z_1z_2$; at a point
where exactly $r\in\{1,2,3\}$ components meet, the germ of $D_Y$ is
$\{z_0\cdots z_{r-1}=0\}$. Thus $D_Y$ is reduced and its local branches are
smooth coordinate hyperplanes. By the branch description and uniqueness of
analytic normalization
\cite[Proposition~8.4.1 and Theorem~8.4.2]{GrauertRemmert}, its local
normalization is the disjoint union of these branches. On chart overlaps the branches are the
restrictions of the globally defined toric components $E_v$, so the local
normalizations glue to \eqref{eq:DY-normalization}.
\end{proof}

\begin{proposition}
\label{prop:lifted-central-quotient}
Fix $a\in\Delta$ and set
\[
 \widetilde D_Y=\coprod_{v\in\Z^2}E_v.
\]
The action of $\Z^2$ on $D_Y$ lifts to a free properly discontinuous action
by biholomorphisms on $\widetilde D_Y$.  Its analytic orbit space is
biholomorphic to $E_0$.
\end{proposition}

\begin{proof}
Each biholomorphism $\Psi_\lambda|_{D_Y}$ sends $E_v$ to
$E_{v+B_0\lambda}$ and lifts uniquely to the normalization. The uniqueness
of these lifts gives the group law. Since $B_0\in\GL(2,\Z)$, the map
\[
 \Z^2\times E_0\longrightarrow\widetilde D_Y,\qquad
 (\lambda,x)\longmapsto\widetilde\Psi_{\lambda,a}(x)
\]
is a biholomorphism, with $\Z^2$ discrete. It identifies the lifted action
with translation on the first factor. The action is therefore free and
properly discontinuous, and its analytic quotient is $E_0$.
\end{proof}

\begin{theorem}
\label{thm:relative-cusp}
The quotient
\[
 \mathcal N_{0,\Delta}
   =(Y_\varepsilon\times\Delta)/\Z^2
\]
is a connected Hausdorff second-countable complex manifold, and every
parameter slice of $\mathcal N_{0,\Delta}\to\Delta$ is connected.  The induced
map
\[
 F_0:\mathcal N_{0,\Delta}
     \longrightarrow D_\varepsilon\times\Delta
\]
is proper, and the projection to $\Delta$ is a holomorphic submersion.  The
inverse image
\[
 W_\Delta:=F_0^{-1}(\{0\}\times\Delta)
\]
is a reduced normal-crossing divisor in $\mathcal N_{0,\Delta}$.  For every
$a\in\Delta$, the fibre
\[
 W_a:=F_0^{-1}(0,a)
\]
is a compact, reduced, irreducible, nonnormal surface with normal-crossing
singularities, and its normalization is isomorphic to the compact toric
surface $E_0$.
\end{theorem}

\begin{proof}
Freeness and proper discontinuity are
\cref{prop:cusp-free,prop:cusp-proper-discontinuous}; hence Lee's holomorphic
quotient theorem \cite[Theorem~1.16]{LeeComplexManifolds} makes the quotient a
complex manifold.  The dense torus of $Y$ meets
$Y_\varepsilon$ in
\[
 (\C^*)^2\times D_\varepsilon^*,
\]
which is connected.  This subset is dense in $Y_\varepsilon$, so
$Y_\varepsilon$ is connected as the closure of a connected subset.  The
product $Y_\varepsilon\times\Delta$ is therefore connected, and its continuous
surjective quotient $\mathcal N_{0,\Delta}$ is connected.  For fixed
$a\in\Delta$, the slice upstairs is the connected space $Y_\varepsilon$, and
its quotient by the restricted $\Z^2$-action is the corresponding parameter
slice of $\mathcal N_{0,\Delta}$; hence that slice is connected as well.
For a compact set $K\subset D_\varepsilon\times\Delta$,
\cref{prop:compact-fundamental} supplies a compact set $C_K$ meeting every
orbit over $K$. Its image in the quotient is exactly $F_0^{-1}(K)$, which is
therefore compact. Thus $F_0$ is proper. The parameter projection upstairs
is the ordinary projection from $Y_\varepsilon\times\Delta$, hence is a
submersion.  Since the quotient map is locally
biholomorphic, the descended map is also a submersion.

In every maximal toric chart, the equation $q=0$ is
$z_0z_1z_2=0$ by \eqref{eq:q-local}.  The quotient map is locally
biholomorphic, so the same equation describes
\[
 W_\Delta=F_0^{-1}(\{0\}\times\Delta).
\]
Hence $W_\Delta$ is a reduced normal-crossing divisor in the four-dimensional
manifold $\mathcal N_{0,\Delta}$.  Fix $a\in\Delta$.  Intersecting these local
charts with the parameter slice $a=\mathrm{constant}$ shows that
$W_a=F_0^{-1}(0,a)$ is a reduced normal-crossing complex surface.  It is
compact because $F_0$ is proper.  Set $D_Y=q^{-1}(0)$.

By \cref{lem:toric-central-normalization}, the normalization of $D_Y$ is
$\nu_Y:\widetilde D_Y=\coprod_v E_v\to D_Y$.
The lifted action in \cref{prop:lifted-central-quotient} gives
\begin{equation}
\label{eq:normalized-central-quotient}
 \widetilde D_Y/\Z^2\xrightarrow{\ \cong\ }E_0.
\end{equation}
The equivariant map $\nu_Y$, followed by the covering $p:D_Y\to W_a$,
therefore induces a holomorphic map
\begin{equation}
\label{eq:E0-to-W}
 \nu_a:E_0\longrightarrow W_a.
\end{equation}
It is surjective because every component $E_v$ has a translate equal to $E_0$.

Choose a sufficiently small evenly covered open set $U\subset W_a$ and
one sheet $\widetilde U\subset D_Y$, so that
\[
 p^{-1}(U)=\coprod_{\lambda\in\Z^2}\Psi_\lambda(\widetilde U).
\]
Every orbit in the normalization over $U$ has a unique representative over
$\widetilde U$. Hence the lifted action identifies
\[
 \nu_a^{-1}(U)\simeq\nu_Y^{-1}(\widetilde U),
\]
and under this identification $\nu_a$ is the local normalization map
$\nu_Y^{-1}(\widetilde U)\to\widetilde U\simeq U$ from
\cref{lem:toric-central-normalization}. These identifications show that
$\nu_a$ is the normalization of $W_a$. Since $E_0$ is irreducible and
$\nu_a$ is surjective, $W_a$ is irreducible.

Finally, the triangulation has edges.  At a point of the toric orbit
corresponding to an edge, the local equation of the central fibre is
$z_0z_1=0$.  The quotient map is locally biholomorphic, so $W_a$ contains a
nonempty double-curve stratum with the same local equation.  The local
normalization has two points over a general point of this stratum, and
therefore $\nu_a$ is not an isomorphism.  Thus $W_a$ is nonnormal.
\end{proof}

On the cusp horodisc, use the adapted fibre coordinates in which
$\Lambda_{\tor}$ is represented by the last two columns of the period matrix.
Set
\begin{equation}
\label{eq:Eexp}
 E_{\exp}(z,\zeta,a)
 =\bigl(\e^{2\pi\ii\zeta_1},
        \e^{2\pi\ii\zeta_2},
        q(z),a\bigr).
\end{equation}
Translations by $\Lambda_{\tor}$ disappear under the first two exponentials.
For $\lambda\in\Lambda$, let $\bar\lambda$ denote its image in
$\overline\Lambda\cong\Z^2$.  The decomposition
\eqref{eq:cusp-decomposition} gives
\begin{equation}
\label{eq:Eexp-equivariance}
 E_{\exp}(z,\zeta+\Omega_a(z)\lambda,a)
 =\Psi_{\bar\lambda}(E_{\exp}(z,\zeta,a)).
\end{equation}
Under the cusp deck transformation the adapted vertical coordinate is
unchanged because $R_{g_0}=I_2$, while $q(g_0z)=q(z)$.  Hence
$E_{\exp}(g_0z,\zeta,a)=E_{\exp}(z,\zeta,a)$.  Equation
\eqref{eq:cusp-deck-period} is the corresponding compatibility of the period
lattice with this deck action.

\begin{proposition}
\label{prop:cusp-gluing}
For the disc $D_0=D_\varepsilon$ and the smaller horodisc
$\mathcal H_{0,\varepsilon}$ defined in
\eqref{eq:smaller-horodisc}, Equation \eqref{eq:Eexp} descends to a jointly
holomorphic biholomorphism
\begin{equation}
\label{eq:G0}
 G_0:\mathcal J_\Delta|_{D_0^*\times\Delta}
 \xrightarrow{\ \cong\ }
 \mathcal N_{0,\Delta}|_{D_0^*\times\Delta}.
\end{equation}
\end{proposition}

\begin{proof}
We first identify the relevant part of the global core.  Let
\[
 \widetilde{\mathcal J}_{0,\Delta}^*
 =
 \widetilde{\mathcal J}_\Delta
 \big|_{\mathcal H_{0,\varepsilon}\times\Delta}.
\]
If
$g\widetilde{\mathcal J}_{0,\Delta}^*
 \cap\widetilde{\mathcal J}_{0,\Delta}^*\ne\varnothing$,
then the projections to $U^\circ$ give
\[
 g\mathcal H_{0,\varepsilon}
 \cap\mathcal H_{0,\varepsilon}\ne\varnothing.
\]
The inherited separation property proved immediately after
\eqref{eq:smaller-cusp-quotient} therefore implies
$g\in\langle g_0\rangle$. Thus
$\widetilde{\mathcal J}_{0,\Delta}^*$ is precisely invariant under
$\langle g_0\rangle$. Since the orbit maps are open, Lee's local quotient
charts \cite[Theorem~1.16]{LeeComplexManifolds}, together with
\eqref{eq:smaller-cusp-quotient}, give
\begin{equation}
\label{eq:cusp-core-restriction}
 \widetilde{\mathcal J}_{0,\Delta}^*/\langle g_0\rangle
 \xrightarrow{\ \cong\ }
 \mathcal J_\Delta|_{D_0^*\times\Delta}.
\end{equation}
Thus it remains only to identify the quotient on the left with the punctured
toric quotient.

First quotient $\C^2$ by the invariant lattice
$\Lambda_{\tor}=\Z\langle\widehat w,\widehat\delta\rangle$.  In the adapted
coordinates its two generators are the last two columns of $\Omega_a$, so the
map
\[
 (\zeta_1,\zeta_2)\longmapsto
 (\e^{2\pi\ii\zeta_1},\e^{2\pi\ii\zeta_2})
\]
identifies that quotient with $(\C^*)^2$.  The remaining quotient lattice is
$\overline\Lambda\cong\Z^2$, and
\eqref{eq:Eexp-equivariance} says exactly that its action is
\eqref{eq:Psi-dense}.  Thus $E_{\exp}$ induces a bijection after the two
lattice quotients are taken.

The cusp deck transformation leaves $q$ and the exponential fibre coordinates
unchanged, so taking its quotient replaces the additive coordinate $s$ by the
punctured-disc coordinate $q=\e^{2\pi\ii s}$.  To construct the inverse near
a point of the quotient, choose local logarithms of the two nonzero torus
coordinates and a local logarithm of $q$.  Dividing the first two logarithms
by $2\pi\ii$ gives $\zeta$, and dividing the third by $2\pi\ii$ gives a lift
$s$ of the base point.  Changing a fibre logarithm adds a generator of
$\Lambda_{\tor}$; changing the logarithm of $q$ changes $s$ by an integer and
therefore changes the chosen point of
$\mathcal H_{0,\varepsilon}$ by a power of the deck transformation $g_0$.  A different representative under the residual
$\Z^2$ action changes $\zeta$ by the remaining periods, by
\eqref{eq:Eexp-equivariance}.  Hence the local inverse is independent of all
choices after passing to the quotients.  The local logarithms are holomorphic,
so this inverse is holomorphic.  Both directions are jointly holomorphic in
$a$.  Combining this identification with
\eqref{eq:cusp-core-restriction} proves \eqref{eq:G0}.
\end{proof}

\subsection{Global compactification and smooth type}
\label{sec:global-gluing}

For a fixed parameter, this gluing construction and its independence of the
auxiliary choices are due to Alp\"oge~\cite[Construction~6.1,
Theorem~6.2, and Proposition~6.3]{AlpogeS6}.  We carry out the same gluing over
$\Delta$ and record the compatibility as the parameter disc and the common
cusp cutoff vary.

Choose the three discs $D_0,D_1,D_2\subset B$ with pairwise disjoint closures.
The four base opens
\[
 B^\circ,\quad D_0,\quad D_1,\quad D_2
\]
cover $B$, and the only nonempty overlaps involving a filling disc are
$D_j^*=D_j\cap B^\circ$.

Glue the disjoint union
\[
 \mathcal J_\Delta\sqcup
 \mathcal N_{0,\Delta}\sqcup
 \mathcal N_{1,\Delta}\sqcup
 \mathcal N_{2,\Delta}
\]
by $G_0$ from \eqref{eq:G0} and by $G_j$ from \eqref{eq:Gj}.  Since the
filling discs are disjoint, there are no nontrivial triple overlaps among
three different pieces; hence the cocycle condition is automatic once the
three displayed maps are fixed.

\begin{lemma}
\label{lem:hausdorff-gluing}
The glued topological space $\mathfrak X_\Delta$ is Hausdorff and
second-countable.  The quotient charts determine a unique complex manifold
structure on it.
\end{lemma}

The filling discs are disjoint, so every gluing class is either a singleton
or a two-point class joining the core to one filling.  Each piece therefore
embeds openly, and disjoint neighbourhoods inside one piece have disjoint
open saturations. The complete verification is given in Appendix~\ref{app:deferred-verifications}.

The local maps glue to a holomorphic map
\begin{equation}
\label{eq:global-relative-map}
 \mathfrak f:\mathfrak X_\Delta\longrightarrow B\times\Delta.
\end{equation}

\begin{theorem}
\label{thm:relative-compactification}
The map \eqref{eq:global-relative-map} is proper.  The parameter projection
\[
 \pi=\operatorname{pr}_\Delta\circ\mathfrak f:
 \mathfrak X_\Delta\longrightarrow\Delta
\]
is a proper holomorphic submersion.  Its fibre over $a$ is the compactification obtained from $\Omega_a$ by the
two logarithmic fillings and the cusp filling with the chosen valid cusp
cutoff.
\end{theorem}

\begin{proof}
On the four members of the base open cover, the restrictions of
$\mathfrak f$ are $p^\circ,F_0,F_1,F_2$, all of which are proper by
\cref{prop:relative-core,prop:relative-finite-filling,thm:relative-cusp}.
The base $B\times\Delta$ is metrizable, and the local properness just proved
shows that $\mathfrak f$ is proper.

Let $K\subset\Delta$ be compact.  Since $B$ is compact, $B\times K$ is
compact.  Hence
\[
 \pi^{-1}(K)=\mathfrak f^{-1}(B\times K)
\]
is compact, so $\pi$ is proper.

On the core, the parameter projection is a submersion by
\cref{prop:relative-core}.  On the finite fillings, it is a submersion by
\cref{prop:relative-finite-filling}; this remains true at the multiple central
fibres because the free quotient map is locally biholomorphic and the
parameter $a$ is unchanged by the group action.  On the cusp filling, it is a
submersion by \cref{thm:relative-cusp}; this remains true at the
normal-crossing fibre because upstairs the total space is the smooth product
$Y_\varepsilon\times\Delta$ and the action is free.  These local parameter
maps agree under the gluing biholomorphisms, so $\pi$ is a holomorphic
submersion everywhere.

All group actions preserve $a$, and all gluing maps are defined by formulas
whose restriction at a fixed parameter is the corresponding fixed-parameter
formula.  Thus the fibre over $a$ is the compactification made with the
present parameter disc and the present cutoff.  Its independence of these
choices is established in \cref{prop:choice-independence}.
\end{proof}

\begin{proposition}
\label{prop:choice-independence}
The relative compactification has the following compatibility properties.
\begin{enumerate}[label=\textup{(\roman*)}]
\item Fix a relatively compact parameter disc $\Delta\Subset\mathcal A$.
      If $0<\varepsilon'<\varepsilon$ are two valid cusp cutoffs
      for $\Delta$ in the sense of Definition~\ref{def:valid-cusp-cutoff}, then the two relative
      compactifications are canonically biholomorphic over $B\times\Delta$.
      The biholomorphism is the identity on the core and on the two finite
      fillings.
\item Let $\Delta_1,\Delta_2\Subset\mathcal A$ be parameter discs, with any valid cusp cutoffs in the sense of
      Definition~\ref{def:valid-cusp-cutoff} chosen in the two constructions.  Over
      $\Delta_1\cap\Delta_2$, their relative compactifications are
      canonically biholomorphic.  These identifications satisfy the cocycle
      identity on triple intersections.
\end{enumerate}
Consequently, for every $a\in\mathcal A$, the fibre together with its
distinguished map to $B$ is independent, up to the canonical biholomorphisms
over $B$ constructed above, of the parameter disc and of the cusp cutoff.
\end{proposition}

Nested valid cutoffs give invariant open embeddings of cusp models, and the
punctured identifications are restrictions of the same exponential formula.
Two parameter discs are compared locally using one smaller common cutoff;
the restriction identities make these comparisons independent of that
cutoff and give the cocycle law. The complete verification is given in Appendix~\ref{app:deferred-verifications}.

\begin{definition}
\label{def:intrinsic-Xa}
For $a\in\mathcal A$, let $(X_a,f_a)$ be the fibre pair obtained from any
relatively compact disc $\Delta\Subset\mathcal A$ containing $a$ and any
valid cusp cutoff, whose existence follows from
\cref{lem:valid-cutoff-exists}.  Proposition~\ref{prop:choice-independence}
identifies all such choices by canonical biholomorphisms over $B$ satisfying
the cocycle law.  We use $(X_a,f_a)$ for any fixed representative.
\end{definition}

Every $X_a$ is connected: the core and each of the three connected
fillings meet along a nonempty punctured annulus.  Moreover, $X_a$ is compact
and $f_a(X_a)$ contains the dense open set $B^\circ$; hence
$f_a:X_a\to B$ is surjective.

\begin{corollary}
\label{cor:global-smooth-type}
For any two admissible parameters $a,a'\in\mathcal A$, the fixed-parameter
compactifications $X_a$ and $X_{a'}$ are diffeomorphic.
\end{corollary}

\begin{proof}
Choose a disc $\Delta\Subset\mathcal A$ containing $a$ and $a'$.  The
projection of the resulting relative
compactification is a proper holomorphic submersion by
\cref{thm:relative-compactification}; hence its fibres over $a$ and $a'$
are diffeomorphic by Ehresmann's theorem
\cite[pp.~154--155]{Ehresmann}.
The comparison maps of \cref{prop:choice-independence,def:intrinsic-Xa}
identify these fibres with the fixed representatives $X_a$ and $X_{a'}$.
\end{proof}

\begin{lemma}
\label{lem:alpoge-member}
Let $(X^{\mathrm A},f^{\mathrm A})$ be the compact pair of
Alp\"oge~\cite[Construction~6.1]{AlpogeS6}, with its specified twists
$v_1=\widehat\gamma+2\widehat u-4\widehat w$ and
$v_2=-\widehat\gamma-3\widehat u+3\widehat w$, and untranslated cusp
gluing.
For its chosen admissible period function $\beta^{\mathrm A}$, there is a
unique constant $a_{\mathrm A}\in\mathcal A$ for which
$\beta^{\mathrm A}=\beta_0+a_{\mathrm A}$. For this parameter there is an
isomorphism of pairs
\[
 (X^{\mathrm A},f^{\mathrm A})\cong(X_{a_{\mathrm A}},f_{a_{\mathrm A}})
 \quad\text{over }B.
\]
\end{lemma}

\begin{proof}
Use the same oriented $(3,4,\infty)$ triangle and the same affine
normalization of its quotient: the order-three, order-four, and cusp
points are $0,1,\infty$, respectively. This identifies the marked
coverings and their generators $g_1,g_2,g_0$. The integral matrices
\eqref{eq:A12} are Alp\"oge's dual monodromy matrices in the ordered
basis $(\widehat\gamma,\widehat u,\widehat w,\widehat\delta)$.
His $\tau$ and $\mu$ have the generator laws
\eqref{eq:period-generator-laws} and
\eqref{eq:mu-generator-laws}, and the same
cusp normalization as ours.  The uniqueness assertions in
\cite[Theorem~3.4(i),(ii)]{AlpogeS6} identify them.  His
$\beta^{\mathrm A}$ and our $\beta_0$ satisfy
\eqref{eq:beta-generator-laws}, and both
$\beta+\tau$ extend across the cusp. Their difference descends to a
holomorphic function on $B=\Pone$, hence is a constant
$a_{\mathrm A}$:
\[
 \beta^{\mathrm A}=\beta_0+a_{\mathrm A},\qquad
 \Pi^{\mathrm A}=\Omega_{a_{\mathrm A}}.
\]
This also fixes the vertical matrices $R_g$, since they are determined
by the same period matrix and integral monodromy.

Alp\"oge's admissibility condition is
$\mathscr D_{a_{\mathrm A}}(z)<0$ for every $z\in U$. By
\cref{prop:period-data-existence}, this invariant function is continuous
on $B\setminus\{p_0\}$ and tends to $-\infty$ at $p_0$; it therefore
attains a strictly negative maximum. Since
$\mathscr D_{a_{\mathrm A}}=\mathscr D_0+\Im a_{\mathrm A}$, we obtain
$M_0+\Im a_{\mathrm A}<0$, or $a_{\mathrm A}\in\mathcal A$.

The four local pieces now agree at $a_{\mathrm A}$. On the smooth core
they are quotients of the same marked period family by the same deck
action. At $p_j$, the vectors in \eqref{eq:twists} equal the two source
vectors; hence the finite actions
\eqref{eq:finite-affine-action} and the logarithmic translations
\eqref{eq:sigma-ja} agree with Alp\"oge's finite fillings and gluing
maps. At the cusp, both constructions use the triangles
$\{v,v+e_1,v+e_2\}$ and
$\{v+e_1,v+e_2,v+e_1+e_2\}$ for $v\in\Z^2$, the same
$B_0=\bigl(\begin{smallmatrix}0&1\\-1&0\end{smallmatrix}\bigr)$,
and the same decomposition
$Z_{a_{\mathrm A}}=sB_0+C_{a_{\mathrm A}}(q)$. Thus the toric lattice
action \eqref{eq:Psi-action} and the exponential gluing
\eqref{eq:Eexp} agree. In particular, neither construction inserts a
fibre translation at the cusp. In the discrete notation of
\cite[Section~6.4]{AlpogeS6}, these facts say
$(\ell_0,\ell_1,\ell_2)=(0,1,-1)$, with the last two values given by
\eqref{eq:twist-gamma}.

Take the finite linearizing discs and the cusp cutoff to be common
smaller choices in the two constructions. On these choices, the
identity on each of the four pieces respects every gluing map, hence
descends to a biholomorphism of pairs over $B$. Restriction to smaller
discs removes only punctured collars already covered by the core.
Changing a finite logarithm branch changes its fibre translation by an
integral period; changing the cusp logarithm changes $C$ by an
integral multiple of $B_0$, whose exponential action is unchanged.
Thus the comparison is independent of these choices, consistently with
\cite[Proposition~6.3]{AlpogeS6}.
The intrinsic representative $X_{a_{\mathrm A}}$ is independent of
its parameter disc and cusp cutoff by
\cref{prop:choice-independence,def:intrinsic-Xa}, giving the stated
isomorphism.
\end{proof}

\begin{corollary}
\label{cor:sphere-type}
Every compact threefold $X_a$ in the admissible family is diffeomorphic
to $S^6$.
\end{corollary}

\begin{proof}
By \cref{lem:alpoge-member}, $X_{a_{\mathrm A}}$ is Alp\"oge's
threefold, which is diffeomorphic to $S^6$ by
\cite[Theorem~8.1]{AlpogeS6}. The smooth type is constant throughout
$\mathcal A$ by \cref{cor:global-smooth-type}.
\end{proof}


\section{Cyclic coverings and Oka maps}
\label{nm:sec:relative-oka}

An open subset is called analytic-Zariski open if its complement in the
stated ambient manifold is a closed analytic subset.

For the family in \cref{thm:relative-compactification}, set
\[
 \mathcal S_j=(\mathfrak f_\Delta^{-1}(\{p_j\}\times\Delta))_{\red},
 \qquad \mathcal U_j=\mathfrak X_\Delta\setminus\mathcal S_j.
\]
Here $(\cdot)_{\red}$ denotes the underlying reduced complex space.
The two complements form an analytic-Zariski cover of
$\mathfrak X_\Delta$.  The normalized base maps are
\[
 \rho_4(b)=\frac{1}{1-b^4}:\C_b\longrightarrow\Pone\setminus\{0\},
 \qquad
 \rho_3(b)=\frac{b^3}{b^3-1}:\C_b\longrightarrow\Pone\setminus\{1\}.
\]
For $(m,j)=(4,1)$ or $(3,2)$, let $\mathcal M_m$ be the normalization of
the following fibre product in the category of reduced complex spaces:
\[
 \mathcal U_j\mathbin{\times}_{(\Pone\setminus\{p_j\})\times\Delta}
                 (\C_b\times\Delta).
\]
There are maps
\begin{equation}\label{nm:eq:finite-cover}
 r_m:\mathcal M_m\longrightarrow\mathcal U_j,
 \qquad h_m:\mathcal M_m\longrightarrow\C_b\times\Delta.
\end{equation}
The root base change removes the remaining multiple fibre.  The local
model from \cref{sec:relative-compactification} is the quotient of
$t=s^m$ by a free affine cyclic action of order $m$, so $r_m$ is a finite
unramified holomorphic covering of degree $m$ and $\mathcal M_m$ is
smooth.

\begin{lemma}\label{nm:lem:unwrap}
For $m=3,4$ there is a connected regular unramified holomorphic covering
\[
 q_m:\mathcal Q_m\longrightarrow\mathcal M_m
\]
with deck group $\Z$, whose restriction to a smooth torus fibre is the
covering corresponding to the subgroup $\Lambda_0=\ker\gamma$ of the period
lattice.  For every $a_*\in\Delta$
there is a biholomorphism
\begin{equation}\label{nm:eq:product}
 \mathcal Q_m\simeq Q_{m,a_*}\times\Delta
\end{equation}
over $\C_b\times\Delta$, where $Q_{m,a_*}=\mathcal Q_m|_{a_*}$.
\end{lemma}

\begin{proof}
On an ordinary or normalized finite-root chart, replacing $\Lambda$ by
$\Lambda_0$ gives the required local cyclic cover.  Its periods are
independent of $a$ by \eqref{nm:eq:invariant-direction}.  On a cusp chart
the same cover is
\begin{equation}\label{nm:eq:local-cusp-cover}
 (Y_\varepsilon/\langle\Psi_{(0,1)}\rangle)\times\Delta
 \longrightarrow (Y_\varepsilon\times\Delta)/\Z^2,
\end{equation}
because $E_{21}(0,n)^t=0$.  All transition maps preserve $\Lambda_0$ and
the positive generator of $\Lambda/\Lambda_0$.  Their lifting obstruction
is therefore a constant $\Z$-valued \v Cech $2$-cocycle on $\C_b$, which
vanishes since $H^2(\C,\Z)=0$.  The local covers consequently glue to the
connected regular covering $q_m$.

It remains to remove the parameter dependence of the lifted transitions.
The ordinary periods restricted to $\Lambda_0$ are independent of $a$.
In the finite-root charts, the logarithmic translations in
\cref{lem:finite-conjugation} vary linearly with $a$ along $e_2$;
on the cusp charts the retained subgroup satisfies $E_{21}(0,n)^t=0$.
The lifted transitions therefore commute with the action $\Theta$ in
\eqref{nm:eq:invariant-direction}, which extends across the fillings.
After choosing logarithms on the overlaps, the transition at $a$
differs from the transition at $a_*$ by
\[
 \Theta_{\exp(2\pi\ii(a-a_*)h_{\alpha\beta}(b))},
 \qquad h_{\alpha\beta}\in\OO(V_\alpha\cap V_\beta),
\]
where $h_{\alpha\beta}$ is independent of $a$.  The transition
identities and the effectiveness of $\Theta$ give
\[
 \exp\!\left(2\pi\ii(a-a_*)
 (h_{\alpha\beta}+h_{\beta\gamma}-h_{\alpha\gamma})\right)=1
 \qquad\text{for every }a\in\Delta.
\]
It follows that $h_{\alpha\beta}+h_{\beta\gamma}=h_{\alpha\gamma}$.
Cartan's Theorem~B now gives $h_{\alpha\beta}=b_\alpha-b_\beta$.
Rescaling the local trivializations by
$\Theta_{\exp(-2\pi\ii(a-a_*)b_\alpha)}$ makes every transition independent
of $a$ and proves \eqref{nm:eq:product}.
\end{proof}

For a fixed $a_*\in\cA$, the cyclic cover $Q_{m,a_*}$ contains a conic band
whose base is an elliptic surface. In this band construction, write $z$ for
the root coordinate $b$ in \eqref{nm:eq:finite-cover}, in agreement with the
elliptic-surface equations below.

\begin{lemma}\label{nm:lem:actual-two-row-band}
For $m=3,4$, the cyclic covering $Q_{m,a_*}$ of
\cref{nm:lem:unwrap} has an analytic-Zariski open subset $Y_m$.
There is a smooth analytic surface $\mathcal E_m^{\rm an}$ proper over $\C_z$,
with precisely
the $m$ nodal fibres $z^m=1$, and a holomorphic line bundle
$L_m\to\mathcal E_m^{\rm an}$ for which $Y_m$ is a proper small
resolution of
\[
 H_m=\{A\otimes B=\Delta_m(z)\}
       \subset L_m\oplus L_m^{-1},
 \qquad
 \Delta_3=z^3-1,\quad \Delta_4=1-z^4.
\]
Its exceptional locus consists of one complete $\Pone$ over the
node of each nodal fibre.
\end{lemma}

\begin{proof}
For $(m,j)=(3,2),(4,1)$, write
$U_j=X_{a_*}\setminus S_{j,a_*}$. The normalized root maps
$\rho_3(z)=z^3/(z^3-1)$ and $\rho_4(z)=1/(1-z^4)$ give the finite
unramified covers $N_m=M_{m,a_*}\to U_j$ of degree $m$ from
\cref{nm:eq:finite-cover}. The further
cover $Q_{m,a_*}\to N_m$ retains precisely
\[
 \Lambda_0=\ker\gamma
   =\Z\langle\widehat u,\widehat w,\widehat\delta\rangle .
\]
Over an ordinary base chart its fibre is
\begin{equation}
 \C^2/\Z\langle(\tau,\mu),(1,0),(0,1)\rangle .
 \label{nm:eq:band-ordinary-fibre}
\end{equation}
At a cusp $z=\zeta$, $\zeta^m=1$, the local model is
$Y_\varepsilon/\langle\Psi_{(0,1)}\rangle$, as in
\cref{nm:eq:local-cusp-cover}.  Write its torus coordinates as
$(x_1,x_2,q)$.  Since $B_0(0,1)^t=(1,0)^t$, the generator acts by
\begin{equation}
 (x_1,x_2,q)\longmapsto
  \bigl(q\alpha(q)x_1,\beta(q)x_2,q\bigr),
 \label{nm:eq:band-cusp-action}
\end{equation}
where $\alpha$ and $\beta$ are holomorphic units on the full cusp
disc.  The other deck generator has
$B_0(1,0)^t=(0,-1)^t$, and hence shifts the second, or row,
index by one.

\emph{Step 1: the elliptic quotient.}
The fan of \cref{prop:periodic-conical-fan} has vertices $(i,j,1)$.
Projection
\begin{equation}
 (i,j,r)\longmapsto(i,r)
 \label{nm:eq:band-fan-projection}
\end{equation}
takes its cones into those of the one-dimensional Tate fan. On the
chart between columns $i$ and $i+1$, write
\begin{equation}
 u_i=x_1q^{-i},\qquad v_i=x_1^{-1}q^{i+1},\qquad u_iv_i=q.
 \label{nm:eq:band-base-node}
\end{equation}
The projected column transformation $T$ is induced on the dense torus by
\begin{equation}
 T(x_1,q)=(\vartheta(q)x_1,q),\qquad
 \vartheta(q)=q\alpha(q).
 \label{nm:eq:projected-column-action}
\end{equation}
It sends the $i$-th node chart to the $(i+1)$-st by
\begin{equation}
 (u_i,v_i)\longmapsto
   (u_{i+1},v_{i+1})=(\alpha(q)u_i,\alpha(q)^{-1}v_i).
 \label{nm:eq:node-chart-transition}
\end{equation}
After shrinking the cusp disc, $t=\vartheta(q)=q\alpha(q)$ is a
holomorphic coordinate with $|t|<1$, since $\vartheta'(0)=\alpha(0)\ne0$.
The coordinate changes
\[
 \widetilde u_i=\alpha(q)^{-i}u_i,\qquad
 \widetilde v_i=\alpha(q)^{i+1}v_i,\qquad
 \widetilde u_i\widetilde v_i=t
\]
are restrictions of $(x_1,q)\mapsto(x_1,t)$ on the dense torus, so they
glue on the node charts. In these coordinates, $T$ shifts the $i$-th chart
to the $(i+1)$-st while preserving $\widetilde u_i,\widetilde v_i$.
Thus this is the standard analytic Tate construction with $b=1$
\cite[Subsection~2.1 and Remark~2.2]{ArguzTate}. Its column action is free
and properly discontinuous, and the quotient is a Hausdorff complex surface
with one irreducible nodal rational central fibre.

The quotient is proper over a smaller closed cusp disc. Indeed, over a
compact subset of that disc, the closure in the toric model of the annular
region
\[
 |\vartheta(q)|\leq |x_1|\leq1
\]
is contained in the union of two adjacent node charts and has compact closure:
indeed, $|x_1|\le1$ bounds $u_0=x_1$, while
$|x_1|\ge|q\alpha(q)|$ bounds $v_0=q/x_1$ because $\alpha$ is a unit.  This
closure meets every column orbit over $q\ne0$ and every limiting orbit over
$q=0$. Its image in the column quotient is exactly the inverse image of the
chosen compact base set and is compact.  This proves properness. The local
equation $u_iv_i=q$ shows that the total space is smooth.

On the ordinary charts the free $x_2$-scaling quotient of
\eqref{nm:eq:band-ordinary-fibre} is the elliptic curve
$\C/(\Z+\tau\Z)$.  The transitions used to construct $Q_{m,a_*}$
commute with this scaling and their first coordinate is independent
of $x_2$; the cusp transitions induce
\eqref{nm:eq:band-base-node}. They therefore glue the ordinary
quotients and the nodal cusp quotients to a smooth surface
$p_m:\mathcal E_m^{\rm an}\to\C_z$. Properness is local on the
target and follows on the ordinary and cusp charts as just checked.
Its fibres are connected; the only singular fibres occur at
$z^m=1$, each has one node and multiplicity one.  In particular
the fibration is relatively minimal.

At each cusp retain the central toric divisors in rows $j=0,1$
and delete those in every other row.  Retain all noncentral points,
and call the resulting subset of $Q_{m,a_*}$ $Y_m$.
The deleted divisors form a locally finite analytic family before
the column quotient.  The quotient action only translates columns,
so they remain locally finite after quotient; there are only
finitely many cusps.  Their union is consequently closed analytic.
Thus $Y_m$ is analytic-Zariski open in $Q_{m,a_*}$.

\emph{Step 2: descent of the line coordinate.}
On the two retained rows the characters
\begin{equation}
 A=x_2,\qquad B^{\rm tor}=q/x_2
 \label{nm:eq:band-conic-characters}
\end{equation}
have nonnegative pairing with every retained ray and hence extend
holomorphically. Over $\Delta_m(z)\ne0$, the free $x_2$-scaling gives
a principal $\C^*$-bundle. Its associated line bundle extends over the
cusp charts as follows. Use $A=x_2$ as the local line coordinate over
the ordinary quotient charts and the node charts
\eqref{nm:eq:band-base-node}. The scaling-equivariant transitions of
$Q_{m,a_*}$ read
\[
 A_\alpha=g_{\alpha\beta}A_\beta,
\]
with $g_{\alpha\beta}$ a holomorphic unit on the corresponding overlap
in the elliptic base. Ordinary-to-cusp overlaps lie in $q\ne0$.
Between adjacent column charts, the transition is the unit $\beta(q)$
from \eqref{nm:eq:band-cusp-action}, which is holomorphic and nonzero
on the full cusp disc. The transition identities for $Q_{m,a_*}$ give
$g_{\alpha\beta}g_{\beta\gamma}=g_{\alpha\gamma}$ on every triple
overlap. These units are consequently the transition functions of a
holomorphic line bundle $L_m$ on all of $\mathcal E_m^{\rm an}$.

The global discriminant determines the dual coordinate. On $A\ne0$,
this coordinate is
\[
 B=\frac{\Delta_m(z)}{A}.
\]
Near a node, the quotient
\begin{equation}
 \delta(z)=\frac{\Delta_m(z)}q
 \label{nm:eq:discriminant-unit}
\end{equation}
is a holomorphic unit, and hence
$B=\delta(z)B^{\rm tor}$ extends across the retained toric charts.  Since
$\Delta_m$ is global and $A_\alpha=g_{\alpha\beta}A_\beta$, one has
$B_\alpha=g_{\alpha\beta}^{-1}B_\beta$.  Thus $B$ is the fibre coordinate of
$L_m^{-1}$, and the characters give a global holomorphic map to
\[
 H_m=\{A\otimes B=\Delta_m(z)\}\subset L_m\oplus L_m^{-1}.
\]

\emph{Step 3: the global contraction.}
On a node chart, \eqref{nm:eq:band-base-node} and
\eqref{nm:eq:discriminant-unit} give
\[
 AB=\delta u_iv_i=\Delta_m(z).
\]
The two triangles in each unit square select the same ruling as the
incidence resolution
\[
 \{v_is=At,\ Bs=\delta u_it\}
 \subset\{AB=\delta u_iv_i\}\times\Pone_{[s:t]}.
\]
Its two affine charts are
\begin{equation}
 v_i=Ap,\quad B=\delta u_ip,\qquad
 u_i=\delta^{-1}Br,\quad A=v_ir,\qquad r=1/p.
 \label{nm:eq:band-small-resolution}
\end{equation}
On an overlap of node charts write
\[
 u_\alpha=\kappa_{\alpha\beta}u_\beta,\qquad
 v_\alpha=\kappa_{\alpha\beta}^{-1}v_\beta,\qquad
 A_\alpha=g_{\alpha\beta}A_\beta,\qquad
 B_\alpha=g_{\alpha\beta}^{-1}B_\beta,
\]
where both $\kappa_{\alpha\beta}$ and $g_{\alpha\beta}$ are holomorphic
units.  For adjacent column charts,
$\kappa_{\alpha\beta}=\alpha(q)$ and
$g_{\alpha\beta}=\beta(q)$ by
\eqref{nm:eq:node-chart-transition} and
\eqref{nm:eq:band-cusp-action}.  The projective transition is
\begin{equation}
 [s_\alpha:t_\alpha]
   =[s_\beta:(\kappa_{\alpha\beta}g_{\alpha\beta})^{-1}t_\beta].
 \label{nm:eq:band-projective-transition}
\end{equation}
Indeed, the two incidence equations transform as
\[
 v_\alpha s_\alpha
  =\kappa_{\alpha\beta}^{-1}v_\beta s_\beta
  =A_\alpha t_\alpha,
 \qquad
 B_\alpha s_\alpha
  =g_{\alpha\beta}^{-1}B_\beta s_\beta
  =\delta u_\alpha t_\alpha.
\]
Thus the node-coordinate and line-coordinate changes are both included, and
the local incidence resolutions glue.  Their affine charts are exactly the
two retained toric rows in \eqref{nm:eq:band-small-resolution}; away from the
nodes the contraction is an isomorphism.  Hence they give a global map
$Y_m\to H_m$.

Near each of the finitely many nodes the incidence space is a closed analytic
subspace of $H_m\times\Pone$, so its projection is proper; away from those
neighbourhoods the contraction is an isomorphism.  Properness is local on the
target, and therefore $Y_m\to H_m$ is proper.  Its fibre over each node is the
complete $\Pone$ obtained from the two affine exceptional lines in
\eqref{nm:eq:band-small-resolution}. This proves the assertion.
\end{proof}

The proper elliptic fibration $p_m:\mathcal E_m^{\rm an}\to\C_z$ is
relatively minimal and has no multiple fibres by
\cref{nm:lem:actual-two-row-band}. The section theorem of
Buzzard and Lu~\cite[Lemma~3.8]{BuzzardLu} gives a holomorphic section $O$.

The period data identify the elliptic surface underlying the conic
band with an explicit algebraic model. The relevant affine hypersurfaces are
Danielewski surfaces:
\begin{align}
 W_3&=\{uv=w(w+2\ii)\},&z&=(4u^2-v)/3,\label{nm:eq:base3}\\
 W_4&=\{uv=w(4w^2-3)\},&z&=(u-v)/(2\ii).\label{nm:eq:base4}
\end{align}

\begin{lemma}\label{nm:lem:algebraic-band-base}
There are fibre-preserving biholomorphisms
$\mathcal E_m^{\rm an}\simeq E_m^{\rm an}$, where $E_m$ is the
Weierstrass surface over $\mathbf A^1_z$
\begin{equation}
 E_3:\ y^2=4x^3-3zx-1,\qquad
 E_4:\ y^2=4x^3-3x-z^2,
 \label{nm:eq:realized-elliptic-surfaces}
\end{equation}
with its section at infinity included.  The complement of this section in
$E_m$ is the affine Danielewski surface $W_m$ displayed in
\eqref{nm:eq:base3} or \eqref{nm:eq:base4}, respectively.
\end{lemma}

\begin{proof}
With the section $O$, Nakayama's Weierstrass-model theorem
\cite[Theorem~2.1]{NakayamaWeierstrass} gives a proper bimeromorphic
morphism from $\mathcal E_m^{\rm an}$ to its Weierstrass model over
$\C$. Every fibre is irreducible and meets $O$, so no fibre component
is contracted and this morphism is an isomorphism. The fundamental
line bundle is trivial over $\C$, so the model has a global equation
\begin{equation}
 y^2=4x^3-g_2(z)x-g_3(z),\qquad g_2,g_3\in\cO(\C).
 \label{nm:eq:band-general-Weierstrass}
\end{equation}

 Alp\"oge's normalization $j(\tau)=1728t$, together with the root maps
 $t=\rho_3(z)=z^3/(z^3-1)$ and
 $t=\rho_4(z)=1/(1-z^4)$ from
 \cref{nm:eq:finite-cover}, gives
\begin{equation}
 j_3(z)=1728\,\frac{z^3}{z^3-1},
 \qquad j_4(z)=\frac{1728}{1-z^4}.
 \label{nm:eq:band-j-functions}
\end{equation}
For $m=3$, $g_3$ cannot vanish at a smooth fibre, since $j_3$
there never equals $1728$; at a nodal fibre both $g_2$ and
$g_3$ are nonzero.  It is therefore a holomorphic unit.
Its sixth root exists on simply connected $\C$, and scaling
$(x,y)$ makes $g_3=1$.  The formula
\begin{equation}
 j=1728\,\frac{g_2^3}{g_2^3-27g_3^2}
 \label{nm:eq:band-j-formula}
\end{equation}
then gives $g_2^3=27z^3$; a remaining constant root-of-unity
scaling makes $g_2=3z$.  For $m=4$, $j_4$ never vanishes at
a smooth fibre, so $g_2$ is a unit.  A fourth-root scaling
makes $g_2=3$, after which
\eqref{nm:eq:band-j-formula} gives $g_3^2=z^4$.
The remaining constant choice makes $g_3=z^2$.
This proves \eqref{nm:eq:realized-elliptic-surfaces}.  Its
discriminants are $27(z^3-1)$ and $27(1-z^4)$, as required.

The complement of $O=[0:1:0]$ in $E_3$ has coordinates
\[
 u=x,\quad w=y-\ii,\quad v=4x^2-3z,
 \qquad uv=w(w+2\ii),\quad z=(4u^2-v)/3.
\]
For $E_4\setminus O$, the coordinates
\[
 u=y+\ii z,\quad v=y-\ii z,\quad w=x,
 \qquad uv=w(4w^2-3),\quad z=(u-v)/(2\ii)
\]
give the asserted models.
\end{proof}

Write $\varrho_m:Y_m\to H_m$ for the small resolution in
\cref{nm:lem:actual-two-row-band}. Since the conifold points lie in
$\{A=B=0\}$, the principal open
\begin{equation}\label{nm:eq:principal-conic-open}
 Y_{m,A}:=\varrho_m^{-1}(H_m\cap\{A\ne0\})\simeq L_m^\times
\end{equation}
is a holomorphic principal $\C^*$-bundle over $\mathcal E_m^{\rm an}$.

In the Danielewski coordinates, the algebraic open
$Q^\circ=\{u\ne0\}\subset W_m$ is isomorphic to
$\C_u^*\times\C_w$. Since $Q^\circ$ is Stein and has the homotopy type
of a circle, the Oka--Grauert principle
\cite[Theorem~5.3.1(ii)]{ForstnericBook} makes
$L_m|_{Q^\circ}$ holomorphically trivial.
Choose such a trivialization, so that $A$ is a scalar fibre coordinate.
For a fixed $w_0\in\C$, write $\eta=w-w_0$. Then
\begin{equation}\label{eq:band-product-open}
 T=Q^\circ\setminus\{w=w_0\}\simeq(\C^*)^2_{u,\eta},
 \qquad
 V=Y_{m,A}|_T\simeq(\C^*)^3_{u,\eta,A}.
\end{equation}
The complement of $T$ in $E_m$ is algebraic, so $V$ is
analytic-Zariski open in $Y_{m,A}$.

\begin{proposition}\label{nm:prop:one-fibre-oka}
For every $a_*\in\cA$, the principal opens $Y_{m,A}$, the bands $Y_m$,
the cyclic covers $Q_{m,a_*}$, the two complements
$X_{a_*}\setminus S_{1,a_*}$ and $X_{a_*}\setminus S_{2,a_*}$, and
$X_{a_*}$ are Oka.
\end{proposition}

\begin{proof}
The product $V\simeq(\C^*)^3$ in \eqref{eq:band-product-open} is Oka.
Analytic extension \cite[Theorem~1.9]{AnDuGuoWangXieLAP} first makes
$Y_{m,A}$ Oka and then $Y_m$ Oka, since each contains the preceding
space as an analytic-Zariski open subset. One band $Y_m$ is itself
analytic-Zariski open in $Q_{m,a_*}$, so the same theorem makes
$Q_{m,a_*}$ Oka.  For $(m,j)=(4,1)$ and $(3,2)$, the composite
\[
 Q_{m,a_*}\longrightarrow M_{m,a_*}\longrightarrow
 X_{a_*}\setminus S_{j,a_*}
\]
is an unramified holomorphic covering. Invariance of the Oka property
under holomorphic coverings \cite[Proposition~2.8]{ForstnericSurvey}
therefore makes both complements Oka. Finally, either complement is a nonempty
analytic-Zariski open
subset of $X_{a_*}$, so analytic extension proves that $X_{a_*}$ is Oka.
\end{proof}

\begin{proof}[Proof of \cref{nm:thm:oka-local}]
The absolute assertions follow from
\cref{nm:prop:one-fibre-oka}.  Fix $(m,j)=(4,1)$ or $(3,2)$.  By
\cref{nm:lem:unwrap}, the composite covering over $\mathcal U_j$ has the
form
\begin{equation}\label{nm:eq:product-cover}
 P_m:Q_{m,a_*}\times\Delta\longrightarrow\mathcal U_j,
 \qquad \pi\circ P_m=\pr_\Delta.
\end{equation}
Here $\pr_\Delta$ is the projection onto $\Delta$. Since $Q_{m,a_*}$ is
Oka, the fibre-bundle Oka principle \cite[Corollary~5.7]{KusakabeMaps}
gives this projection the parametric Oka property with approximation and
interpolation (POPAI). The holomorphic covering $P_m$ is a fibre bundle
with discrete Oka fibres, so the same principle gives it POPAI. Since a
covering is a Serre fibration, $P_m$ is a surjective Oka map. Therefore
\cite[Corollary~5.6]{KusakabeMaps} gives POPAI for $\pi|_{\mathcal U_j}$.

The sets $\mathcal U_1$ and $\mathcal U_2$ form an analytic-Zariski open
cover of $\mathfrak X_\Delta$, so
\cite[Corollary~5.13]{KusakabeMaps} gives POPAI for $\pi$.  The relative
compactification theorem makes $\pi$ a proper holomorphic submersion, and
Ehresmann's theorem makes it a Serre fibration.  Hence $\pi$ is an Oka
map.
\end{proof}


\section{Point modifications and band complements}
\label{sec:point-complements}

For the point-deletion argument, fix $a\in\cA$ and abbreviate $X=X_a$.  Write
$M_m=M_{m,a}$ and $Q_m=Q_{m,a}$ for the finite and cyclic covers from
\cref{nm:eq:finite-cover,nm:lem:unwrap}, and let $\Xi_m\simeq\Z$ be the
deck group of $Q_m\to M_m$.  All noncentral points belong to every
translated two-row band; this common principal smooth locus is denoted by
$\mathcal K_m$.

If $R\subset M_m$ is finite and $\widehat R\subset Q_m$ is its complete
inverse image, then for every two-row band $Y\simeq Y_m$ there is a
decomposition
\begin{equation}\label{eq:core-toroidal-decomposition}
 \widehat R\cap Y=D_Y\sqcup F_Y.
\end{equation}
Here $D_Y$ is a finite union of full $\Xi_m$-orbits through points of
$\mathcal K_m$, whereas $F_Y$ is finite.  Indeed, the lift of each point
is one deck orbit.  A core orbit lies in every band, while a toroidal
orbit meets a two-row band in at most two points because the deck
generator shifts the row index by one.

Every orbit occurring in $D_Y$ lies over one of finitely many curves
\begin{equation}\label{eq:common-core-carrier}
 C_\nu=\{z=z_\nu\}\subset E_m,
 \qquad \Delta_m|_{C_\nu}=\Delta_m(z_\nu)\in\C^*.
\end{equation}
The following elementary choice will separate their logarithmic lifts.

\begin{lemma}\label{lem:proper-character}
For finitely many reduced algebraic curves in $(\C^*)^2$, there is a
monomial $M(u,v)=u^rv^s$ whose restriction to the normalization of every
irreducible component is proper and nonconstant, and hence finite.
\end{lemma}

\begin{proof}
Complete each normalized affine curve to a smooth projective curve.  At
every boundary point, the two orders of $u$ and $v$ do not both vanish;
otherwise the branch would extend to a point of the original closed curve
inside $(\C^*)^2$.  Choose $(r,s)\in\Z^2$ outside the finitely many
rational lines on which
$r\operatorname{ord}u+s\operatorname{ord}v=0$.  The character then tends
to $0$ or $\infty$ at every boundary branch, so it is proper and
nonconstant. It is therefore finite.
\end{proof}

\begin{proposition}\label{prop:punctured-band}
Suppose $D\subset\mathcal K_m$ is a finite union of full $\Xi_m$-orbits and
$F\subset Y_m$ is finite. Then
\[
 Y_m\setminus(D\cup F)
 \quad\text{and}\quad
 Y_{m,A}\setminus(D\cup F)
\]
are Oka.
\end{proposition}

\begin{proof}
Use the product open $V$ from \eqref{eq:band-product-open} and write
$E=(D\cup F)\cap V$. A full deck orbit is closed and discrete in $Q_m$,
so $E$ is closed and discrete in $V$. Its complete inverse image
$\widetilde E\subset\C^3_{x,y,\sigma}$ under
\[
 (x,y,\sigma)\longmapsto(u,\eta,A)
       =(\e^x,\e^y,\e^\sigma)
\]
is also closed and discrete. Apply \cref{lem:proper-character} to the
components of $C_\nu\cap T$ and choose the resulting nonconstant
character $M=u^r\eta^s$, with $(r,s)\ne(0,0)$. If there are no such
components, take $(r,s)=(1,0)$. The projection
\[
 (x,y,\sigma)\longmapsto(rx+sy,\sigma)
\]
maps $\widetilde E$ to a closed discrete subset of $\C^2$. Indeed,
bounded projected logarithms keep $M$ and the scalar coordinate $A$
in compact annuli. Properness of $M$ on the carrier curves and the
finiteness of $F$ confine the base points of $E$ to a compact subset
of $T$. The corresponding
points of $E$ therefore lie in a compact subset of
$V=T\times\C^*$, so only finitely many occur. The projected logarithms
of each point lie in a product of two discrete lattices.

The preceding projection discreteness and the Rosay--Rudin projection
criterion \cite[Theorem~3.9]{RosayRudin} make $\widetilde E$ tame, with the
empty case understood trivially.

The tame-complement theorem
\cite[Proposition~2.24]{ForstnericSurvey} makes
$\C^3\setminus\widetilde E$ Oka. Since $\widetilde E$ is the complete
inverse image of $E$, the exponential map restricts to an unramified
holomorphic covering
$\C^3\setminus\widetilde E\to V\setminus E$. Thus $V\setminus E$ is Oka by
\cite[Proposition~2.8]{ForstnericSurvey}. This common analytic-Zariski open
subset lies in both manifolds displayed in the statement. Applying analytic
extension to each one \cite[Theorem~1.9]{AnDuGuoWangXieLAP} completes the
proof.
\end{proof}

\begin{theorem}\label{thm:semistable-finite-punctures}
For $m=3,4$ and every finite subset $R\subset M_m$, the manifold
$M_m\setminus R$ is Oka.
\end{theorem}

\begin{proof}
Let $\widehat R$ be the complete inverse image in $Q_m$ and choose one
two-row band $Y\simeq Y_m$.  The decomposition
\eqref{eq:core-toroidal-decomposition} and
\cref{prop:punctured-band} show that $Y\setminus\widehat R$ is Oka.
It is analytic-Zariski open in $Q_m\setminus\widehat R$, so analytic
extension makes the latter Oka. The restriction
$Q_m\setminus\widehat R\to M_m\setminus R$ is an unramified holomorphic
covering, so \cite[Proposition~2.8]{ForstnericSurvey} gives the conclusion.
\end{proof}

\begin{proof}[Proof of \cref{thm:intro-point-modifications}]
Let $P\subset X$ be finite and take $U_1=X\setminus S_{1,a}$.  The degree
four cover $M_4\to U_1$ pulls $P\cap U_1$ back to a finite set $R$.
By \cref{thm:semistable-finite-punctures}, $M_4\setminus R$ is Oka.
The restricted degree-four covering $M_4\setminus R\to U_1\setminus P$
therefore makes $U_1\setminus P$ Oka by
\cite[Proposition~2.8]{ForstnericSurvey}. This is analytic-Zariski open in
$X\setminus P$, so \cite[Theorem~1.9]{AnDuGuoWangXieLAP} proves that
$X\setminus P$ is Oka.

Let $E$ be the exceptional divisor of $\Bl_PX\to X$.  Its complement is
biholomorphic to $X\setminus P$ and is analytic-Zariski open in the
blow-up.  A final application of
\cite[Theorem~1.9]{AnDuGuoWangXieLAP} proves that $\Bl_PX$ is Oka.
\end{proof}

The same principal open also gives complements of curves in the band.
Let $C_1,\ldots,C_m\subset Y_m$ be the exceptional curves of the small
resolution.  The closed analytic curve
$Z_m^0=\{A=B=0\}\subset H_m$ has complete inverse image
\[
 \mathfrak A_m=\varrho_m^{-1}(Z_m^0),
\]
which contains the exceptional curves and the proper transforms of the
remaining zero-axis components.

\begin{theorem}\label{thm:partial-exceptional}
Suppose $D,F$ satisfy the hypotheses of \cref{prop:punctured-band} and
$I\subset\{1,\ldots,m\}$. Then
\[
 Y_m\setminus\left(D\cup F\cup\bigcup_{i\in I}C_i\right)
\]
is Oka.
\end{theorem}

\begin{proof}
The Oka manifold $Y_{m,A}\setminus(D\cup F)$ from
\cref{prop:punctured-band} is disjoint from all exceptional curves and is
analytic-Zariski open in the displayed complement.  Apply
\cite[Theorem~1.9]{AnDuGuoWangXieLAP}.
\end{proof}

In particular, deleting any subcollection of the exceptional curves from
$Y_m$ preserves the Oka property.

\begin{theorem}\label{thm:axis-deletion}
The manifold $Y_m\setminus\mathfrak A_m$ is Oka.
\end{theorem}

\begin{proof}
The principal conic open $Y_{m,A}$ in
\eqref{nm:eq:principal-conic-open} is Oka and analytic-Zariski open in
$Y_m\setminus\mathfrak A_m$.  Apply
\cite[Theorem~1.9]{AnDuGuoWangXieLAP}.
\end{proof}

\begin{corollary}\label{cor:exotic-projective-threefold}
For every $a\in\cA$ and $p\in X_a$, the blow-up
$\Bl_pX_a$ is an Oka complex manifold diffeomorphic to $\PP^3$, but it is
not biholomorphic to the standard projective threefold.
\end{corollary}

\begin{proof}
The Oka assertion follows from
\cref{thm:intro-point-modifications}.  Since $X_a$ is diffeomorphic to
$S^6$, its smooth point blow-up is diffeomorphic to $\PP^3$
\cite{HuckleberryKebekusPeternell}.  If $E\simeq\PP^2$ is the exceptional
divisor, the canonical-bundle formula gives
\[
 K_{\Bl_pX_a}\simeq\pi^*K_{X_a}\otimes\mathcal O_{\Bl_pX_a}(2E).
\]
Alp\"oge's smooth-type theorem gives $H^2(X_a,\Z)=0$
\cite[Theorem~8.1]{AlpogeS6}, while $[E]$ generates
$H^2(\Bl_pX_a,\Z)\simeq\Z$.  Thus
$c_1(T_{\Bl_pX_a})=-2[E]$ has divisibility $2$, whereas
$c_1(T_{\PP^3})=4h$ has divisibility $4$.  Hence the two complex
structures are not biholomorphic.
\end{proof}

\begin{corollary}\label{cor:smooth-curve-oka1}
If $C\subset X_a$ is a smooth compact complex curve, possibly
disconnected, then $\operatorname{Bl}_C X_a$ is Oka--$1$.  The same holds
after any finite sequence of blow-ups along smooth centres.
\end{corollary}

\begin{proof}
This follows from \cref{nm:prop:one-fibre-oka} and the blow-up invariance of
the Oka--$1$ property
\cite[Theorem~3.1]{ForstnericLarussonOka1}.
\end{proof}

\begin{proposition}\label{prop:local-conifold-curve}
For $R=\operatorname{Tot}(\cO_{\Pone}(-1)\oplus\cO_{\Pone}(-1))$ with zero
section $C\simeq\Pone$, one has
\[
 \operatorname{Bl}_C R\simeq
 \operatorname{Tot}(\cO_{\Pone\times\Pone}(-1,-1)),
\]
and this smooth algebraic threefold is algebraically elliptic and Oka.
\end{proposition}

\begin{proof}
The blow-up of a vector bundle along its zero section is the tautological
line bundle on its projectivization.  Here
$\mathbb P(\cO(-1)\otimes\C^2)=\Pone\times\Pone$, giving the displayed
isomorphism. The line bundle $\cO_{\Pone\times\Pone}(-1,-1)$ is trivial
over each of the four standard affine charts $\mathbf A^2$ of
$\Pone\times\Pone$. Its total space is therefore covered by four Zariski
opens isomorphic to $\mathbf A^3$, each algebraically elliptic. Gromov's
algebraic localization lemma
\cite[Corollary~2.2]{KalimanZaidenbergCubic} makes this total space
algebraically elliptic. Gromov's theorem gives the Oka property
\cite{GromovOka}.
\end{proof}


\appendix

\section{Gluing and choice-independence}
\label{app:deferred-verifications}

This appendix supplies the two global compatibility arguments deferred from
\cref{sec:relative-compactification}.

\begin{proof}[Proof of \cref{lem:hausdorff-gluing}]
Write
\[
 C=\mathcal J_\Delta,\qquad N_j=\mathcal N_{j,\Delta}.
\]
Let
\[
 O_j^C=C|_{D_j^*\times\Delta},
 \qquad
 O_j^N=N_j|_{D_j^*\times\Delta}.
\]
For uniform notation, let
$\phi_j:O_j^C\to O_j^N$ be $G_0$ when $j=0$ and $G_j^{-1}$ when
$j=1,2$.  The gluing relation is generated by
\[
 x\sim\phi_j(x)\qquad(x\in O_j^C).
\]
Because the discs $D_0,D_1,D_2$ are pairwise disjoint, the three open sets
$O_j^C$ are pairwise disjoint.  No point of one filling piece is identified
with a point of another filling piece.  It follows directly that every
equivalence class is either a singleton or a two-point set
\[
 \{x,\phi_j(x)\}
\]
for a unique index $j$ and a unique $x\in O_j^C$.  In particular, an
equivalence class contains at most one point from each individual piece.
Hence the canonical map from every piece to the glued quotient is injective.

The image of each piece is open.  For example, the saturation of the core is
\[
 C\ \sqcup\ O_0^N\ \sqcup\ O_1^N\ \sqcup\ O_2^N,
\]
which is open in the disjoint union, and the saturation of $N_j$ is
$N_j\sqcup O_j^C$.  The canonical injection of a piece is therefore a
homeomorphism onto an open subset.  Since the glued space is covered by four
open subsets homeomorphic to second-countable spaces, the union of four
countable bases is a countable basis for the glued space.  The complex charts
on the pieces have holomorphic transition functions because the only
cross-piece transitions are the biholomorphisms $\phi_j$ and
$\phi_j^{-1}$.  Thus they determine a unique complex-manifold structure once
Hausdorffness has been proved.

Let $x\ne y$ be points of the glued space.  The local maps to
$B\times\Delta$ agree on the overlaps, so they determine a continuous map
$\mathfrak f$.  If $\mathfrak f(x)\ne\mathfrak f(y)$, choose disjoint
neighbourhoods of their images in the Hausdorff base and take inverse images.

Suppose
$\mathfrak f(x)=\mathfrak f(y)=(b,a)$.  If $b\in D_j$, represent both points
in the filling piece $N_j$; when $b\in D_j^*$ this is possible using
$\phi_j$.  If $b$ lies outside the three filling discs, represent both points
in the core $C$.  Thus there is one Hausdorff piece $M$ containing distinct
representatives $\widetilde x,\widetilde y$ of $x,y$.

Choose disjoint open neighbourhoods $U,V\subset M$ of
$\widetilde x,\widetilde y$.  Their saturations in the disjoint union are
open.  Explicitly, if $M=C$, then
\[
 \operatorname{Sat}(U)
 =
 U\ \sqcup\ \bigsqcup_j\phi_j(U\cap O_j^C),
\]
and if $M=N_j$, then
\[
 \operatorname{Sat}(U)
 =
 U\ \sqcup\ \phi_j^{-1}(U\cap O_j^N);
\]
the same formulas hold for $V$.  These two saturations are disjoint.  Indeed,
an intersection point would determine one equivalence class meeting both
$U$ and $V$.  That class would then contain two representatives in the same
piece $M$, contrary to the one-representative property proved above and the
disjointness of $U,V$.  The quotient images of the two open saturations are
therefore disjoint open neighbourhoods of $x$ and $y$.  Hence the glued space
is Hausdorff.
\end{proof}
The remaining compatibility concerns the cusp cutoff and the parameter
disc used in the relative construction.

\begin{proof}[Proof of \cref{prop:choice-independence}]
We first fix $\Delta$ and compare two radii
$0<\varepsilon'<\varepsilon$.  For every valid cusp cutoff $\rho$ for $\Delta$, write
\[
 \mathcal N_{0,\Delta}^{(\rho)}
   =(Y_\rho\times\Delta)/\Z^2,
 \qquad
 \mathfrak X_\Delta^{(\rho)}
\]
for the cusp quotient and the resulting global compactification, and write
\[
 G_0^{(\rho)}:
 \mathcal J_\Delta|_{D_\rho^*\times\Delta}
 \xrightarrow{\ \cong\ }
 \mathcal N_{0,\Delta}^{(\rho)}|_{D_\rho^*\times\Delta}
\]
for the punctured cusp biholomorphism constructed in
\cref{prop:cusp-gluing}.  All these objects are obtained by restricting the
same action formula and the same exponential map.
The core $\mathcal J_\Delta$ is dense in each compactification:
the omitted fibres are locally given by $s_j^{m_j}=0$ or
$z_0z_1z_2=0$. Since the compactifications are Hausdorff by
\cref{lem:hausdorff-gluing}, two continuous maps to the same compactification
that agree on the core agree everywhere. In particular, a comparison map
restricting to the identity on the core is unique.

The subset $Y_{\varepsilon'}\times\Delta$ is open and invariant in
$Y_\varepsilon\times\Delta$. The quotient map is open, and Lee's local
quotient charts \cite[Theorem~1.16]{LeeComplexManifolds} identify its quotient
with an open subset of the larger quotient. Hence there is an open
biholomorphic embedding
\begin{equation}
\label{eq:cutoff-open-embedding}
 j_{\varepsilon',\varepsilon}:
 \mathcal N_{0,\Delta}^{(\varepsilon')}
 \hookrightarrow
 \mathcal N_{0,\Delta}^{(\varepsilon)}.
\end{equation}
Its image is the inverse image of
$D_{\varepsilon'}\times\Delta$.

Let $W_\Delta^{(\varepsilon)}$ be the central divisor in the larger cusp
quotient.  The two open subsets
\[
 j_{\varepsilon',\varepsilon}
   (\mathcal N_{0,\Delta}^{(\varepsilon')})
 \quad\text{and}\quad
 \mathcal N_{0,\Delta}^{(\varepsilon)}
   \setminus W_\Delta^{(\varepsilon)}
\]
cover $\mathcal N_{0,\Delta}^{(\varepsilon)}$.  On the first one, the
map to $\mathfrak X_\Delta^{(\varepsilon')}$ is the inverse of
\eqref{eq:cutoff-open-embedding}, followed by the inclusion of the smaller
cusp filling.  On the second one, the map is given by
\[
 \mathcal N_{0,\Delta}^{(\varepsilon)}\setminus W_\Delta^{(\varepsilon)}
 \xrightarrow{\ (G_0^{(\varepsilon)})^{-1}\ }
 \mathcal J_\Delta|_{D_\varepsilon^*\times\Delta}
 \longrightarrow
 \mathfrak X_\Delta^{(\varepsilon')},
\]
where the last arrow is the inclusion of the common core.

These two definitions agree on their overlap.  Indeed, both punctured cusp
maps are induced by the same exponential formula \eqref{eq:Eexp}; restricting
that formula from $D_\varepsilon^*$ to $D_{\varepsilon'}^*$ gives
\begin{equation}
\label{eq:G0-radius-compatibility}
 G_0^{(\varepsilon)}|_{D_{\varepsilon'}^*\times\Delta}
 =j_{\varepsilon',\varepsilon}
   \circ G_0^{(\varepsilon')}.
\end{equation}
In the smaller glued space, a core point and its image under
$G_0^{(\varepsilon')}$ represent the same point.  The two local maps
therefore glue to a holomorphic map from the larger cusp filling to the
smaller global compactification.

On the common core and the two finite fillings take the identity.  On every
cusp overlap this identity agrees with the map just constructed, again by
\eqref{eq:G0-radius-compatibility}.  The gluing property therefore gives a
holomorphic map
\[
 \Theta_{\varepsilon,\varepsilon'}:
 \mathfrak X_\Delta^{(\varepsilon)}
 \longrightarrow
 \mathfrak X_\Delta^{(\varepsilon')}
\]
over $B\times\Delta$.  Reversing the roles of the two compactifications gives
an inverse: on the smaller cusp filling it is the open embedding
\eqref{eq:cutoff-open-embedding}, and on the core and finite pieces it is the
identity.  The two composites are the identity on every member of the open
gluing cover.  Hence $\Theta_{\varepsilon,\varepsilon'}$ is a
biholomorphism.  Uniqueness gives, for
$\varepsilon''<\varepsilon'<\varepsilon$,
\[
 \Theta_{\varepsilon',\varepsilon''}
 \circ\Theta_{\varepsilon,\varepsilon'}
 =\Theta_{\varepsilon,\varepsilon''}.
\]
This proves part~\textup{(i)}.

Now consider two parameter discs $\Delta_1,\Delta_2$ and a point
$a_0\in\Delta_1\cap\Delta_2$. Choose a disc
$V\Subset\Delta_1\cap\Delta_2$ containing $a_0$.  Let the two original
valid cusp cutoffs be $\eta_1$ and $\eta_2$.  Restricting from
$\Delta_i$ to $V$ can only decrease the supremum in condition~\textup{(V2)},
so each $\eta_i$ remains valid for $V$ with its original auxiliary cutoff.
Choose
\[
 0<\rho<\min\{\eta_1,\eta_2\}.
\]
By the monotonicity assertion of \cref{lem:valid-cutoff-exists}, $\rho$ is a
valid cusp cutoff for $V$.

Every lattice action, deck action, finite cyclic action, cusp action, and
gluing map preserves the parameter.  Consequently, restricting the
construction over $\Delta_i$ to $V$ gives exactly the construction obtained
by applying the same formulas over $V$ with the original cutoff
$\eta_i$.  In particular, the two restrictions have the same core,
the same finite fillings, and the same period functions.  By
part~\textup{(i)}, each is canonically biholomorphic to the construction over
$V$ made with cutoff $\rho$. In the notation of part~\textup{(i)}, the
comparison from the first restriction to the second is
\[
 C_{12}^{V,\rho}
  =(\Theta_{\eta_2,\rho})^{-1}
    \circ\Theta_{\eta_1,\rho}.
\]
Different choices of $\rho$ give comparison maps with the same restriction
to the core, so uniqueness makes them equal. Comparisons constructed over
different parameter neighbourhoods likewise agree on their overlaps and
therefore glue to a canonical biholomorphism over
$\Delta_1\cap\Delta_2$. On a triple overlap, the composite comparison and
the direct comparison both restrict to the identity on the core; uniqueness
gives $C_{23}\circ C_{12}=C_{13}$. This proves part~\textup{(ii)}.

Taking the fibre over a point $a$ in part~\textup{(ii)} proves the final
choice-independence assertion.
\end{proof}


\section{The source complex structure on the six-sphere}
\label{app:source-six-sphere}

The complex structure used throughout this paper originates in
Alp\"oge's construction~\cite[Construction~6.1 and
Theorem~8.1]{AlpogeS6}. Its underlying idea is to complete a family of
complex two-tori over a three-punctured projective line in three
different ways. The finite monodromies at two punctures lead to smooth
multiple fibres; the unipotent monodromy at the third leads to a
reduced, nonnormal toric fibre. The choices of the three completions
also determine the integral topology of the total space. We describe
how these parts fit together, using the notation of
\cref{sec:relative-compactification}.

Take $B=\Pone_t$ with $p_1=0$, $p_2=1$, $p_0=\infty$, and
$B^\circ=B\setminus\{p_0,p_1,p_2\}$. Alp\"oge~\cite[Proposition~2.11]{AlpogeS6}
uses the hyperbolic orbifold of signature $(0;3,4,\infty)$ on $B$.
Its upper-half-plane uniformization $U$ has group
\[
 \Gamma=\langle g_1,g_2\mid g_1^3=g_2^4=1\rangle,
 \qquad g_0=(g_1g_2)^{-1},
\]
where $g_j$ fixes a point over $p_j$ and $g_0$ is the primitive
parabolic element at $p_0$. Removing the two finite-stabilizer orbits
gives a covering $U^\circ\to B^\circ$. Thus the base simultaneously
supplies an order-three point, an order-four point, and a cusp; each
has a distinct local completion of the torus family.

The integral lattice $\Lambda=\Z\langle\widehat\gamma,\widehat u,
\widehat w,\widehat\delta\rangle$ will be the first homology of a
smooth fibre. The representation $g_j\mapsto A_j$ in
\eqref{eq:A12} has $A_1^3=A_2^4=I$, whereas
$A_0=(A_1A_2)^{-1}$ is unipotent. The first-coordinate character
$\gamma:\Lambda\to\Z$ is fixed by both finite monodromies; their
coinvariants leave precisely the circle measured by $\gamma$. At the
cusp, the primitive rank-two lattice
\[
 \Lambda_{\tor}=\ker(A_0-I)=\operatorname{im}(A_0-I)
   =\Z\langle\widehat w,\widehat\delta\rangle
\]
is the lattice that will be exponentiated. On
$\overline\Lambda=\Lambda/\Lambda_{\tor}$, the induced map to
$\Lambda_{\tor}$ has matrix
\[
 B_0=\begin{pmatrix}0&1\\-1&0\end{pmatrix}
\]
in the bases $(\overline{\widehat\gamma},\overline{\widehat u})$
and $(\widehat w,\widehat\delta)$. Its unimodularity is what makes
the cusp's periodic toric construction yield a single irreducible
central fibre. These integral features are the link between the
complex-geometric fillings and Alp\"oge's later topology calculation.

Alp\"oge~\cite[Theorem~3.4]{AlpogeS6} realizes this monodromy by
holomorphic period functions $\tau,\mu,\beta$ on $U$. The first is a
modular parameter satisfying $j(\tau)=1728t\circ\pi$; its
transformations under $g_1,g_2$ are
$\tau\mapsto(\tau-1)/\tau$ and $\tau\mapsto-1/\tau$.
The branching of $j$ at $0$ and $1728$ matches the order-three and
order-four stabilizers: at the latter point the projective modular
action has order two, while its integral lift has order four. The
remaining two functions solve the affine descent equations forced by
the same integral monodromy. In the present notation the resulting
matrix is
\[
 \Omega_a(z)=
 \begin{pmatrix}
  6\mu(z)&\tau(z)&1&0\\
  \beta_0(z)+a&\mu(z)&0&1
 \end{pmatrix},
 \qquad a\in\mathcal A,
\]
where the additive freedom in $\beta$ is recorded by $a$. The
inequality $\mathscr D_a<0$ from \eqref{eq:D-a} gives
$\det_{\R}\Omega_a=\Im\tau\,\mathscr D_a\ne0$. Hence its columns
form a full lattice in $\C^2$. The equivariance
$\Omega_a(gz)=R_g(z)\Omega_a(z)A(g)^{-1}$ makes
$\C^2/\Omega_a(z)\Lambda$ descend from $U^\circ$ to a proper
holomorphic torus family $J_a\to B^\circ$. In particular, the period
functions provide both the complex structure of each fibre and the
coordinate changes needed to carry it around the base.

At $p_j$, a local uniformizing coordinate $s_j$ satisfies
$s_j(g_jz)=\e^{-2\pi\ii/m_j}s_j(z)$, where $(m_1,m_2)=(3,4)$;
the coarse coordinate is $t_j=s_j^{m_j}$. The pulled-back family
extends over $s_j=0$ with a smooth torus fibre. Alp\"oge's
logarithmic transform~\cite[Theorem~5.4]{AlpogeS6} divides this
extension by the affine cyclic action
\[
 (z,\zeta)\longmapsto
 \left(g_jz,R_{g_j}(z)\zeta+
       \Omega_a(g_jz)\frac{v_j}{m_j}\right),
\]
using the invariant vectors
\[
 v_1=\widehat\gamma+2\widehat u-4\widehat w,
 \qquad
 v_2=-\widehat\gamma-3\widehat u+3\widehat w.
\]
The relations $A_jv_j=v_j$, $\gamma(v_1)=1$, and $\gamma(v_2)=-1$ make
the cyclic action free, as verified in \cref{lem:finite-free}.  The quotient
is therefore smooth; its reduced central surface $S_j$
is a free cyclic quotient of a complex two-torus and is bielliptic.
Because the map to the coarse base is $t_j=s_j^{m_j}$, its fibre is
$m_jS_j$, namely $3S_1$ or $4S_2$. Over $s_j\ne0$, translation by
\[
 \frac{\log s_j}{2\pi\ii}\,\Omega_a(z)v_j
\]
conjugates this affine action to the original monodromy action; the branch
independence of the resulting identification with $J_a$ is verified after
\eqref{eq:sigma-ja}.

The cusp requires a different operation. On a simply connected
horodisc above $p_0$, choose a logarithmic coordinate $s$ with
$q=\e^{2\pi\ii s}=1/t$ and $s(g_0z)=s(z)-1$. The period block
$Z_a$ formed by the first two columns of $\Omega_a$ has the form
\[
 Z_a=sB_0+C_a(q),
 \qquad C_a\in\mathcal O(D_0,M_2(\C)),
\]
because $\tau-s$, $\mu$, and $\beta_0+a+\tau$ extend across $q=0$.
Quotienting first by the last two periods, those of
$\Lambda_{\tor}$, turns the fibre coordinates into $(\C^*)^2$.
For $\bar\lambda\in\overline\Lambda$, the remaining period acts
there by multiplication by
\[
 q^{B_0\bar\lambda}
 \exp\bigl(2\pi\ii C_a(q)\bar\lambda\bigr).
\]
The singular factor is thus an integral monomial in $q$, while the
second factor is a holomorphic unit. This is the passage from the
analytic period family to Alp\"oge's toric filling
\cite[Lemma~4.2 and Theorem~4.5]{AlpogeS6}.

More explicitly, triangulate $\R^2$ periodically by the two
unimodular triangles in each unit square and cone the triangles from
the origin in $\R^2\oplus\R$ at height one. The resulting countable
fan is smooth. Its toric threefold has a holomorphic height character
$q$, and on each maximal affine chart this map reads
$q=z_0z_1z_2$. Hence the central divisor upstairs is reduced with
simple normal crossings. Translation of the fan by
$B_0\bar\lambda$, together with the holomorphic unit in the preceding
display, extends the residual $\overline\Lambda$-action over
$q=0$. Alp\"oge proves that, on a sufficiently small cusp disc, this
action is free and properly discontinuous and its quotient is proper
over the disc. Since $B_0\overline\Lambda=\Z^2$, all vertices of the
periodic fan lie in one orbit. The quotient central fibre $W$ is
therefore irreducible and nonnormal, although it has normal-crossing
singularities. Its normalization is the degree-six del Pezzo surface:
the six boundary curves form a hexagon, and opposite sides are
identified in $W$, producing three double curves and two triple
points~\cite[Proposition~4.6]{AlpogeS6}. Exponentiating the fibre
coordinates identifies the toric quotient over $q\ne0$ with the
original family $J_a$ over the punctured cusp disc
\cite[Proposition~4.7]{AlpogeS6}.

The three punctured identifications attach the two logarithmic
fillings and the toric cusp filling to the same smooth torus family.
Alp\"oge~\cite[Construction~6.1 and Theorem~6.2]{AlpogeS6}
shows that the resulting space is a compact Hausdorff complex
threefold $X^{\mathrm A}$ with a proper map to $\Pone$ and exactly
the fibres described above. The affine translations also enter its
topology. In Alp\"oge's notation the cusp gluing is untranslated,
and the character values are
$(\ell_0,\ell_1,\ell_2)=(0,1,-1)$; consequently
\[
 12\ell_0-4\ell_1-3\ell_2=-1.
\]
Alp\"oge~\cite[Theorem~7.17, Corollary~7.29, and
Theorem~8.1]{AlpogeS6} proves that these choices make
$X^{\mathrm A}$ simply connected with the integral homology of
$S^6$, and then identifies its underlying smooth manifold with
$S^6$. Transporting the complex atlas along that diffeomorphism
gives the complex structure on the standard smooth six-sphere.

Our family uses precisely the monodromy matrices, finite twists,
period normalization, and untranslated cusp gluing just described.
The constant $a$ varies the invariant period $\beta_0+a$ through the
admissible half-plane; the member $a=a_{\mathrm A}$ is Alp\"oge's
compact threefold by \cref{lem:alpoge-member}. Over each relatively
compact parameter disc, the jointly holomorphic fillings give a proper
submersion. Together with the choice-independence in
\cref{prop:choice-independence}, \cref{cor:sphere-type} shows that all
members are diffeomorphic to $S^6$. The normalization of $W$ and its paired
boundary curves are the source data used in the cusp filling of
\cref{sec:local-fillings}.

\section*{Acknowledgements}
We thank F. Forstneri\v c for his correspondence with S.-Y. Xie
concerning our Oka result and for bringing G. Etesi's preprint
\cite{EtesiS6} to our attention.

S.-Y. Xie acknowledges partial support from the National Key R\&D
Program of China under Grants No.~2023YFA1010500 and
No.~2021YFA1003100, and from the National Natural Science Foundation
of China under Grants No.~12288201 and No.~12471081, as well as
support from the Xiaomi Young Talents Program.
Y. Kusakabe is supported by JSPS KAKENHI Grant Number JP24K16919.

\section*{AI Use Disclosure}
The research questions, mathematical insights, and principal constructions are due to the authors. GPT-5.6-Sol was used as an auxiliary tool for specific technical problems, calculations, consistency checks, and the drafting and polishing of the exposition. The authors have examined all arguments and take full responsibility for the content of this paper.



\end{document}